\documentclass[12pt]{article}
\usepackage{amsmath,makeidx,amssymb,amscd,amsfonts, amsthm}
\usepackage{graphicx,latexsym}
\usepackage{color,epsfig}
\usepackage[dvips, bookmarksopen, colorlinks,
  linkcolor =blue, urlcolor =red, citecolor =red, menucolor =blue]{hyperref}
\usepackage[left=2.0cm,right=2.0cm,top=2.0cm,bottom= 1.8cm]{geometry}

\newcommand\R{\mathbb{R}}

\newcommand\N{\mathbb{N}}

\newcommand\norm[2][]{{\left\lVert#2\right\rVert_{#1}}}

\newcommand\B{\mathbb{B}}

\newtheorem{theo}{Theorem}
\newtheorem{lemma}[theo]{Lemma}

\newtheorem{df}{Definition}
\newtheorem{remark}{Remark}

\newtheorem{assumption}[theo]{Assumption}

\def\ve{{\varepsilon}}

\def\Ga{{\cal G}_\alpha}
\def\Ger{{\cal G}_{\ve,\alpha}}

\def\.{{\;}}
\def\bv{{\tt{BV}}}
\def\lzo{{L^2(\Omega)}}

\begin{document}
\setcounter{footnote}{1}

\title{On a level-set method for ill-posed problems with piecewise non-constant coefficients}

\author{
 A.~De Cezaro%
 \thanks{Institute of Mathematics Statistics and Physics,
        Federal University of Rio Grande, Av. Italia km 8, 96201-900
        Rio Grande, Brazil
        (\href{mailto:adrianocezaro@furg.br}{\tt adrianocezaro@furg.br}).}
 }

\date{}

\maketitle

\begin{small}
\begin{abstract}
We investigate a level-set type method for solving ill-posed
problems, with the assumption  that the solutions are piecewise, but
not necessarily constant functions with unknown level  sets and
unknown level values.   In order to get stable approximate solutions
of the inverse problem we propose a Tikhonov-type  regularization
approach coupled with a level set framework. We prove the existence
of generalized minimizers for the Tikhonov functional. Moreover, we
prove convergence and  stability for regularized solutions with
respect to the noise level, characterizing the level-set  approach
as a regularization method for inverse problems. We also show the
applicability of the proposed level set method in some interesting
inverse problems arising in elliptic PDE models.
\end{abstract}
\end{small}


\section{Introduction} \label{sec:1}

Since the seminal paper of Santosa \cite{San95}, level set
techniques have been successfully developed and have recently become
a standard technique for solving inverse problems with interfaces
(e.g., \cite{BurOsher05, CBGI09, CVK10, DCMYVO10, DL09, MS08, NLT07,
TO05, DA06}).

In many applications, interfaces represent interesting physical
parameters (inhomogeneities,  heat conductivity between materials
with different heat capacity, interface diffusion problems)  across
which one or more of these physical parameters change value in a
discontinuous manner. The interfaces divide the domain
$\Omega\subset \mathbb{R}^n$ in subdomains $\Omega_j$, with $j=1,
\cdots, k$, of different regions  with specific internal parameter
profiles. Due to the different physical structures of each of these
regions, different mathematical models might be the most appropriate
for describing  them. Solutions of such models represent a free
boundary problem, i.e., one in which interfaces  are also unknown
and must be determined in addition to the solution of the governing
partial differential equation. In general such solutions are
determined by a set of data obtained by indirect  measurements
\cite{BurOsher05, CBGI09, ChungChanTai05,  ChungVese03, CVK10,
LieLysTai06, TaiYao, VeseOsher04}. Applications include image
segmentation problems \cite{ChungVese03, LieLysTai06, TaiYao,
VeseOsher04}, optimal shape designer problems \cite{TaiChan04,
BurOsher05}, Stefan's type problems \cite{BurOsher05}, inverse
potential problems \cite{CLT08, DCLT2010, DCLT09}, inverse
conductivity/resistivity problems \cite{ChungChanTai05, CVK10,
DCMYVO10, Isa90, DA06} among others \cite{BurOsher05, CBGI09, CVK10,
DL09, TaiChan04}.

There is often a large variety of priors information available for
determining the unknown physical parameter, whose characteristic
depends on the given application. In this article, we are interested
in inverse problems that consist in the identification of an unknown
quantity $u \in D(F) \subset X$ that represents all parameter
profiles inside the individual subregions of $\Omega$, from data $y
\in Y$, where $X$ and $Y$ are Banach spaces and $D(F)$ will be
adequately specified in Section~\ref{sec:min-concept}. In this
particular case, only the interfaces between the different regions
and, possibly, the unknown parameter values need to be reconstructed
from the gathered data. This process can be formally described by
the operator equation
\begin{equation}\label{eq:inv-probl}
F(u) \ = \ y \, ,
\end{equation}
where $F: D(F) \subset X \to Y$ is the forward operator.

Neither existence nor uniqueness of a solution to
\eqref{eq:inv-probl} are guarantee. For simplicity, we assume that
for exact data $y \in Y$, the operator equation \eqref{eq:inv-probl}
admit a solution and we do not strive to obtain results on
uniqueness. However, in practical applications, data are obtained
only by indirect measurements of the parameter. Hence, in general,
exact data $y \in Y$ are not known and we have only access to noise
data $y^\delta \in Y$, whose level of noise $\delta
> 0$ are assumed be known \textit{a priori} and satisfies
\begin{equation}\label{eq:noisy-data}
\| y^\delta - y \|_Y \ \le \ \delta \, .
\end{equation}

We assume that  the inverse problem associated with the operator
equation~\eqref{eq:inv-probl} is ill-posed. Indeed, it is the case
in many interesting problems \cite{CVK10, DL09, EngHanNeu96, HR96,
Isa90, TaiChan04, DA06}. Therefore, accuracy of an approximated
solution call for a regularization method \cite{EngHanNeu96}. In
this article we propose a Tikhonov-type regularization method
coupled with a level-set approach to obtain a stable approximation
of the unknown level sets and values of the piecewise (not
necessarily constant) solution of \eqref{eq:inv-probl}.

Many approaches, in particular level set type approaches, have
previously been suggested for such problems . In \cite{ChanTai03,
ChanTai04, ChungChanTai05, FSL05, LS03, San95}, level set approaches
for identification of the unknown parameter $u$ with distinct, but
known, piecewise constant values were investigated.

In \cite{ChanTai04, ChungVese03, CLT08}, level set approaches were
derived to solve inverse problems, assuming that $u$ is defined by
several distinct constant values. In both cases, one needs only to
identify the level sets of $u$, i.e. the inverse problem reduces to
a shape identification problem. On the other hand, when the level
values of $u$ are also unknown, the inverse problem becomes harder,
since, we have to identify both the level sets and the level values
of the unknown parameter $u$. In this situation, the dimension of
the parameter space increases by the number of unknown level values.
Level set approaches to ill- posed  problems with unknown constant
level values appeared before in \cite{DCLT09, DCLT2010, TaiChan04,
TaiHongwei07, TaiYao}. Level set  regularization properties of the
approximated solution for inverse problems are described in
\cite{Burger01, CLT08, DCLT2010, DCLT09, FSL05}.

However, regularization theory for inverse problems where the
components of the parameter $u$ are variable and have
discontinuities have not been well investigated. Indeed, level set
regularization theory applied to inverse problems \cite{DCLT2010,
CLT08, DCLT09} that recover the shape and the values of variable
discontinuous coefficients are unknown to the author. Some early
results in the numerical implementation of level set type methods
were previously used to obtain solutions of elliptic problems with
discontinuous and variable coefficients in \cite{CVK10}.

In this article, we propose a level set type regularization method
to ill-posed problems whose solution is composed by piecewise
components which are not necessarily constants. In other words, we
introduce a level set type regularization method to recover the
shape and the values of variable discontinuous coefficients. In this
framework a level set function is used to parameterized the solution
$u$ of \eqref{eq:inv-probl}. We obtain a regularized solution using
a Tikhonov-type regularization method. Since the level values of $u$
are not-constant and also unknown.

In the theoretical point of view, the advantage of our approach in
relation to \cite{BurOsher05, DCLT09, CLT08, DCLT2010, FSL05, DAL09}
is that we are able to obtain regularized solutions to inverse
problems with piecewise solutions that are more general than those
covered by the regularization methods proposed before. We still
prove regularization properties for the approximated solution of the
inverse problem model \eqref{eq:inv-probl}, where the parameter is a
non-constant piecewise solution. The topologies needed to guarantee
the existence of a minimizer (in a generalized sense) of the
Tikhonov functional (define below in \eqref{eq:mc}) is quite
complicated and differ in some key points from \cite{DCLT2010,
DCLT09, FSL05}. In this particular approach, the definition of
generalized minimizers are quite different from other works
\cite{DCLT09, CLT08, FSL05} (see Definition~\ref{def:quadruple}). As
a consequence,  the arguments used to prove the well-posedeness of
the Tikhonov functional, the stability and convergence of the
regularized solutions of the inverse problem \eqref{eq:inv-probl}
are quite complicated and  need significant improvements (see
Section~\ref{sec:min-concept}).

The main applicability advantage of the proposed  level set type
method compared to those in the literature is that we are able to
apply this method to  problems whose solutions depend of
non-constant parameters. This implies that we are able to  handle
more general and interesting physical problems, where de components
of the desired parameter is not necessarily homogeneous, as those
presented before in the literature \cite{Ber04, CVK10,  DL09,
DCLT09, DCLT2010, TaiChan04, TaiHongwei07, TaiYao, TL07, DAP11}.
Examples of such interesting physical problems are heat conduction
between materials of different heat capacity and conductivity,
interface diffusion processes and  many other types of physical
problems where modeling components are related with embedded
boundaries. See for example \cite{Ber04, DCLT09, CBGI09, CVK10,
DL09, DAP11} and references therein. As a benchmark problem we
analyze two inverse problems modeled  by elliptic PDE's with
discontinuous and variable coefficients.

In contrast, the non-constant characteristics of the level values
impose different types of theoretical problems, since the topologies
where we are able to provide regularization properties of the
approximated solution are more complicated than the  ones presented
before \cite{DCLT09, DCLT2010, TaiChan04, TaiHongwei07, TaiYao}. As
a consequence, the numerical implementations becomes harder than the
others  approaches in the literature \cite{DCLT09, DCLT2010, DAP11,
DAL09}.

The paper is outlined as follows: In
Section~\ref{sec:level-set-formularion}, we formulate the Tikhonov
functional based on the level-set framework. In
Section~\ref{sec:min-concept}, we present the general assumptions
needed in this article and the definition of the set of admissible
solutions. We prove relevant properties about the admissible set of
solutions, in particular, convergence in suitable topologies. We
also present relevant properties of the penalization functional. In
Section~\ref{sec:conv-an}, we prove that the proposed method is a
regularization method to inverse problems, i.e., we prove that the
minimizers of the proposed Tikhonov functional are stable and
convergent with respect to the noise level in the data. In
Section~\ref{sec:num-sol}, a smooth functional is proposed to
approximate minimizers of the Tikhonov functional defined in the
admissible set of solutions. We provide approximation properties and
the optimality condition for the minimizers of the smooth Tikhonov
functional. In Section~\ref{sec:numeric}, we present an application
of the proposed framework to solve some interesting inverse elliptic
problems  with variable coefficients. Conclusions and future
directions are presented in Section~\ref{sec:conclusions}.

\section{The Level-set Formulation}\label{sec:level-set-formularion}

Our starting point is the assumption that the parameter $u$ in
\eqref{eq:inv-probl} assumes two unknown functional values, i.e.,
$u(x) \in \{ \psi^1(x), \psi^2(x) \}$ a.e. in $\Omega \subset
\mathbb R^d$, where $\Omega$ is a bounded set. More specifically, we
assume the existence of a mensurable set $D \subset\subset \Omega$,
with $0<|D|< |\Omega|$, such that $u(x) = \psi^1(x)$ if $x \in D$
and $u(x) = \psi^2(x)$ if $x \in \Omega/D$. With this framework, the
inverse problem that we are interested in this article is the stable
identification of both the shape of $D$ and the value function
$\psi^j(x)$ for $x$ belonging to $D$ and to $\Omega/D$,
respectively, from observation of the data $y^\delta \in Y$.

We  remark that, if $\psi^1(x)=c^1$ and $\psi^2(x)=c^2$  with $c^1$
and $c^2$ unknown constants values, the problem of identifying  $u$
was rigorously studied before in \cite{DCLT09}. Moreover, many other
approaches to  this case appear in the literature, see \cite{DCLT09,
ChanTai04, ChanTai03, BurOsher05} and references therein. Recently,
in \cite{DCLT2010}, a $L^2$ level set approach to identify the level
and constant contrast was investigated.

Our approach differs from the level set methods proposed in
\cite{DCLT2010, DCLT09}, by considering also the identification of
variable unknown levels of the parameter $u$. In this situation,
many topological difficulties appear in order to have a tractable
definition of an admissible set of parameters (see
Definition~\ref{def:quadruple} below). Generalization to problems
with more than two levels are possible applying this approach and
following the techniques derived in \cite{CLT08}. As observed
before, the present level set approach is a rigorous derivation of a
regularization strategy for identification of the shape and
non-constant levels of  discontinuous parameters. Therefore, it can
be applied to physical problems modeled by embedded  boundaries
whose components are not necessarily piecewise constant \cite{FSL05,
BurOsher05, CLT08, DCLT2010, DCLT09}.

In many interesting applications, the inverse problem modeled by
equation~\eqref{eq:inv-probl} is ill-posed. Therefore a
regularization method must be applied in order to obtain a  stable
approximate solution. We propose a regularization method by: First,
introduce a parametrization on the parameter space, using a level
set function $\phi$ that belongs to $H^1(\Omega)$. Note that, we can
identify the distinct level sets of the function $\phi\in
H^1(\Omega)$ with the definition of the Heaviside projector
\begin{align*}
H:\, & H^1(\Omega)  \longrightarrow  L^\infty(\Omega)\\
   &  \,\, \phi  \,\, \, \, \longmapsto \,\,   H(\phi):=\begin{cases}
                                      1 \,\, \mbox{ if } \,\,
                                      \phi(x) > 0\,,\\
                                      0 \,\, \mbox{ other else}\,.
                                    \end{cases}
\end{align*}
Now, from the framework introduced above,  a solution $u$ of
\eqref{eq:inv-probl}, can be represented as
\begin{align} \label{eq:def-P}
u(x) \ = \ \psi^1(x) H(\phi)  +  \psi^2(x) (1 - H(\phi)) \ =: \
P(\phi,\psi^1,\psi^2)(x) \, .
\end{align}
With this notation, we are able to determine the shapes of $D$  as $
\{x \in \Omega\, ;\ \phi(x) > 0 \}$ and $\Omega / D$ as $ \{x \in
\Omega\, ;\ \phi(x) < 0 \}$.

The functional level values $\psi^1(x)$, $\psi^2(x)$ are also
assumed be unknown and they should be  determined as well.

\begin{assumption}\label{ass:1}
We assume that $\psi^1, \psi^2\, \in \B:=\{f: f \mbox{ is measurable
and}\, f(x) \in [m, M]\,, a.e. \,\, in\, \Omega\}$, for some
constant values $m, M$.
\end{assumption}

\begin{remark}
We remark that, $f \in \B$ implies that $f\in L^\infty(\Omega)$.
Since $\Omega$ is bounded $f\in L^1(\Omega)$. Moreover,
\begin{align*}
\int_\Omega f(x) \nabla\cdot \varphi(x) dx  \leq |M| \int_\Omega
|\nabla\cdot (\varphi)(x)|dx \leq |M| \norm{\nabla
\cdot\varphi}_{L^1(\Omega)}\,,\, \forall \varphi \in C_0^1(\Omega,
\mathbb{R}^n)\,.
\end{align*}
Hence $f \in \bv(\Omega)$.
\end{remark}

Note that, in the case that $\psi^1$ and $\psi^2$ assumes two
distinct constant values (as covered by the analysis done in
\cite{BurOsher05, DCLT09, DCLT2010} and references therein) the
assumptions above are satisfied. Hence, the level  set approach
proposed here generalizes the regularization theory developed in
\cite{DCLT09, DCLT2010}.

From \eqref{eq:def-P}, the inverse problem in \eqref{eq:inv-probl},
with data given as in \eqref{eq:noisy-data}, can be abstractly
written as the  operator equation
\begin{equation} \label{eq:inv-probl-mls}
F( P(\phi, \psi^1,\psi^2) ) \ = \ y^\delta \, .
\end{equation}

Once an approximate solution  $(\phi, \psi^1,\psi^2)$ of
\eqref{eq:inv-probl-mls} is obtained, a corresponding solution of
\eqref{eq:inv-probl} can be computed using equation
\eqref{eq:def-P}.

Therefore, to obtain a regularized approximated solution to
\eqref{eq:inv-probl-mls}, we shall consider the least square
approach combined with a regularization term i.e., minimizing the
Tikhonov functional

\begin{equation} \label{eq:mc}
\hat{\Ga}(\phi,\psi^1,\psi^2):= \| F(P(\phi,\psi^1,\psi^2)) -
y^\delta \|^2_{Y}
   + \alpha \Big\{ \beta_1 |H(\phi)|_\bv
   +\beta_2 \| \phi - \phi_0\|^2_{H^1(\Omega)}
   + \beta_3 {\textstyle\sum\limits_{j=1}^2 |\psi^j-\psi^j_0|_\bv} \Big\} \, ,
\end{equation}
where, $\phi_0$ and $\psi^j_0$ represent some \textit{a priori}
information about the true solution $u^*$ of \eqref{eq:inv-probl}.
The parameter $\alpha > 0$ plays the role of a regularization
parameter and the values of $\beta_i\,, i=1,2,3$ act as scaling
factors. In other words, $\beta_i\,, i=1,2,3$ need to be chosen
\textit{a priori}, but independent of the noise level $\delta$. In
practical, $\beta_i\,, i=1,2,3$ can be chosen in order to represent
\textit{a priori} knowledge of features the of the parameter
solution $u$ and/or to improve the numerical algorithm. A more
complete discussion about how to choose $\beta_i\,, i=1,2,3$ are
provided in \cite{DCLT2010, DCLT09, CLT08}.

The regularization strategy in this context is based on $TV-H^1-TV$
penalization. The term on $H^1$-norm acts simultaneously as a
control on the size of the norm of the level set function and a
regularization on the space $H^1$. The term on $\bv$ is a
variational measure of $H(\phi)$. It is well known that the
$BV$-semi-norm acts as a penalizing for the length of the Hausdorff
measure of the boundary of the set $\{x \,:\,\phi(x) > 0\}$  (see
\cite[Chapter 5]{EG92} for details). Finally, the last term on $\bv$
is a variational measure of $\psi^j$ that acts as a regularization
term on the set $\B$. This Tikhonov functional extends the ones
proposed in \cite{DCLT09, CLT08, ChanTai03, ChanTai04, TaiChan04}
(based on $TV$-$H^1$ penalization).

Existence of minimizers for the functional \eqref{eq:mc}, in the
$H^1\times \B^2$ topology does not follow by direct arguments,
since, the operator $P$ is not necessarily continuous in this
topology. Indeed, if $\psi^1 = \psi^2 =\psi$ is a continuous
function at the contact region, then $P(\phi^1,\psi^2, \psi)= \psi$
is continuous and the standard Tikhonov regularization theory to the
inverse problem holds true \cite{EngHanNeu96}. On the other hand, in
the interesting case where $\psi^1$ and $\psi^2$ represents the
level of discontinuities of the parameter  $u$, the analysis becames
more complicated and we need a definition of generalized minimizers
(see Definition~\ref{def:quadruple}) in order to handle with these
difficulties.

\section{Generalized Minimizers} \label{sec:min-concept}

As already observed in \cite{FSL05}, if $D \subset \Omega$ with
$\mathcal{H}^{n-1}(\partial D) < \infty$ where
$\mathcal{H}^{n-1}(S)$ denotes the (n-1)-dimensional
Hausdorff-measure of the set $S$, then the Heaviside operator $H$
maps $H^1(\Omega)$ into the set $$\mathcal{V} :=  \{ \chi_{D} \, ; \
D \subset \Omega \mbox{ measurable},\,:\, \mathcal{H}^{n-1}(\partial
D) < \infty \}\,.$$ Therefore, the operator $P$ in \eqref{eq:def-P}
maps $H^1(\Omega) \times \B^2 $ into the admissible parameter set
$$D(F) := \{ u = q(v,\psi^1,\psi^2)\,;\, v \in \mathcal V
\mbox{ and } \psi^1, \psi^2 \in \B \}\,,$$ where $$q: \mathcal{V}
\times \B^2 \ni (v,\psi^1,\psi^2) \mapsto \psi^1 v+\psi^2(1-v) \in
BV(\Omega)\,.$$

Consider the model problem described in the introduction. In this
article, we assume that:

\begin{minipage}{0.96\textwidth}
\noindent {\bf (A1)} $\Omega \subseteq \R^n$ is bounded with
piecewise $C^1$ boundary $\partial \Omega$.

\noindent {\bf (A2)} The operator $F:D(F)\,\subset L^1(\Omega) \to
Y$ is continuous on $D(F)$ with respect to the
$L^1(\Omega)$-topology.

\noindent {\bf (A3)}
$\ve$, $\alpha$ and $\beta_j\,,j=1,2,3$ denote positive parameters.

\noindent {\bf (A4)} \label{ass:5} Equation (\ref{eq:inv-probl}) has
a solution, i.e. there exists $u_* \in{D(F)}$ satisfying $F(u_*)=y$
and a function $\phi_* \in H^1(\Omega)$ satisfying $|\nabla \phi_*|
\neq 0$, in the neighborhood of $\{\phi_* = 0\}$ such that $H(\phi_*)
= z_*$, for some $z_* \in \mathcal{V}$. Moreover, there exist
functional values $\psi^1_*, \psi^2_* \in \B$ such that
$q(z_*,\psi^1_*,\psi^2_*) = u_*$.
\end{minipage}
\bigskip

\noindent For each $\ve > 0$, we define  a smooth approximation to
the operator $P$ by
\begin{equation} \label{eq:def-Pve}
P_\ve(\phi,\psi^1,\psi^2) \ := \ \psi^1 H_\ve(\phi)  + \psi^2(1 -
H_\ve(\phi))\; ,
\end{equation}
where $H_\ve$ is the smooth approximation to $H$  described by
$$
H_\ve(t) := \left\{
  \begin{array}{rl}
    1 + t/\ve  & \mbox{ for \ } t \in \left[-\ve,0\right] \\
    H(t)       & \mbox{ for \ } t \in \mathbb{R} / \left[-\ve,0\right] \\
\end{array} \right. .
$$

\begin{remark}
It is worth noting that, for any $\phi_k \in H^1(\Omega)$,
$H_\ve(\phi_k)$  belongs to $L^\infty(\Omega)$ and satisfies $0\leq
H_\ve(\phi_k)\leq 1$ a.e. in $\Omega$, for all $\ve
> 0$. Moreover, taking into account that $\psi^j \in \mathbb{B}$,
follows that the operators $q$ and $P_\ve$, as above, are well
defined.
\end{remark}

In order to guarantee the existence of a minimizer of ${\cal G}_\alpha$
defined in \eqref{eq:mc} in the space $H^1(\Omega) \times \B^2$, we
need to introduce a suitable topology such that the functional ${\cal
G}_\alpha$ has a closed graphic. Therefore, the concept of generalized
minimizers (compare with \cite{CLT08, FSL05}) in this paper is:

\begin{df} \label{def:quadruple}
Let the operators $H$, $P$, $H_\ve$ and $P_\ve$ be defined as above
and the positive parameters $\alpha, \beta_j$ and $\ve$ satisfying
the Assumption~\textbf{(A3)}.
\medskip

 A {\bf quadruple} $(z,\phi,\psi^1,\psi^2) \in
L^\infty(\Omega) \times H^1(\Omega) \times \bv(\Omega)^2$ is called
{\bf admissible} when:\newline
 \noindent a) There exists a sequence $\{
\phi_k \}$ of $H^1(\Omega)$-functions satisfying
$\lim\limits_{k\to\infty} \| \phi_k - \phi \|_{L^2(\Omega)} = 0$.

\noindent b) There exists a sequence $\{ \ve_k \} \in \mathbb R^+$
converging to zero such that $\lim\limits_{k\to\infty} \|
H_{\ve_k}(\phi_k)-z \|_{L^1(\Omega)} = 0$.
\medskip

\noindent c) There exist  sequences $\{ \psi^1_k \}_{k\in\mathbb{N}}
\mbox{ and }  \{ \psi^2_k \}_{k\in\mathbb{N}}$ belonging to $\bv \cap
C^\infty(\Omega)$ such that
$$
|\psi^j_k|_\bv \longrightarrow |\psi^j|_\bv\,,\quad j=1,2\,.
$$

\noindent d) A {\bf generalized minimizer} of $\hat{\Ga}$ is
considered to be any admissible quadruple $(z,\phi,\psi^1,\psi^2)$
minimizing
\begin{equation} \label{eq:gzphi}
{\cal{G}}_\alpha(z,\phi,\psi^1,\psi^2) :=
     \norm{ F(q(z,\psi^1,\psi^2)) - y^\delta }_Y^2 + \alpha
     R(z,\phi,\psi^1,\psi^2)
\end{equation}
on the set of admissible quadruples. Here the functional $R$ is
defined by
\begin{eqnarray}\label{def:R}
 R(z,\phi,\psi^1,\psi^2) \ = \ \rho(z,\phi) +
  \beta_3 {\textstyle\sum\limits_{j=1}^2 |\psi^j-\psi^j_0|_\bv} \,
\end{eqnarray}
and the functional $\rho$ is defined as
\begin{eqnarray}\label{eq:rho}
 \rho(z,\phi) := \inf \Big\{ \liminf\limits_{k\to\infty}
         \Big[ \beta_1 |H_{\ve_k} (\phi_k)|_\bv +
         \beta_2 \| \phi_k - \phi_0 \|_{H^1(\Omega)}^2 \Big] \Big\} \, .
         \end{eqnarray}
The infimum in \eqref{eq:rho} is taken over all sequences
$\{\ve_k\}$ and $\{\phi_k \}$ characterizing
$(z,\phi,\psi^1,\psi^2)$ as an admissible quadruple.
\medskip
\end{df}

The convergence $ |\psi^j_k|_\bv \longrightarrow |\psi^j|_\bv$ in Item
c) in Definition~\ref{def:quadruple} is in the sense of variation
measure \cite[Chapter 5]{EG92}. The incorporation of item c) in the
Definition~\ref{def:quadruple} implies the existence of the
$\Gamma$-limit of sequences of admissible quadruples \cite{FSL05,
  AcaVog94}. This appears in the proof of Lemmas~\ref{lemma:auxil},
\ref{lemma:limit-ad} and \ref{lemma:rho-lsc}, where we proove that the
set of admissible quadruples are closed in the defined topology (see
Lemmas~\ref{lemma:auxil} and \ref{lemma:limit-ad}) and in the weak
lower semi-continuity of the regularization functional $R$ (see
Lemma~\ref{lemma:rho-lsc}). The identification of non-constant level
values $\psi^j$ imply in a different definition of admissible
quadruples.

As a consequence, the arguments in the proof of regularization
properties of the level set approach are the principal theoretical
novelty and the  difference between our definition of admissible
quadruples and the ones in \cite{DCLT2010, DCLT09, FSL05}.

\begin{remark}\label{remark-ad}
For $j=1,2$  let $\psi^j\in \B \cap C^\infty(\Omega)$, $\phi\in
H^1(\Omega)$ be such that $|\nabla \phi| \neq 0 $ in the neighborhood
of the level set $\{\phi(x)=0 \}$ and $H(\phi)=z \in \mathcal{V}$.
For each $k\in \N$ set $\psi_k^j=\psi^j$ and $\phi_k =\phi$. Then,
for all sequences of $\{\ve_k\}_{k\in \N}$ of positive numbers
converging to zero, we have
\begin{align*}
\norm{H_{\ve_k}(\phi_k) - z}_{L^1(\Omega)} & =
\norm{H_{\ve_k}(\phi_k) - H(\phi)}_{L^1(\Omega)} =
\int_{(\phi)^{-1}[-\ve_k,0]}\left| 1-\frac{\phi}{\ve_k}\right|dx \\
& \leq \int_{-\ve_k}^0\int_{(\phi)^{-1}(\tau)}1d\tau \leq
meas\{(\phi)^{-1}(\tau)\}\int_{-\ve_k}^01dt \longrightarrow 0\,.
\end{align*}
Here, we use the fact that $|\nabla \phi|\neq 0$ in the neighborhood
of $\{\phi =0\}$ implies that $\phi$ is a local diffeomorphism
together with a co-area formula \cite[Chapter 4]{EG92}. Moreover,
$\{\psi_k^j\}_{k\in \N}$ in $\B \cap C^\infty(\Omega)$ satisfyes
Definition~\ref{def:quadruple}, item c).

Hence, $(z,\phi,\psi^1,\psi^2)$ is an admissible quadruple. In
particular, we conclude from the general assumption above that the
set of admissible quadruple satisfying $F(u)=y$ is not empty.
\end{remark}

\subsection{Relevant Properties of Admissible
Quadruples}\label{subsec:relevant-propoerties}

Our first result is the proof of the continuity properties  of
operators $P_\ve$, $H_\ve$ and $q$ in suitable topologies. Such
result will be necessary in the subsequent analysis.

We start with an auxiliary lemma that is well known (see for example
\cite{DZ94}). We present it here for the sake of completeness.

\begin{lemma}\label{lemma:0}
Let $\Omega$ be a measurable subset of $\mathbb{R}^n$ with finite
measure.

If $(f_k) \in \B$ is a convergent sequence in $L^p(\Omega)$ for some
$p$, $1\leq p<\infty$, then it is a convergent sequence in
$L^p(\Omega)$ for all $1\leq p < \infty$.

In particular Lemma~\ref{lemma:0} holds for the sequence $z_k: =
H_\varepsilon(\phi_k)$.
\end{lemma}

\begin{proof}
See \cite[Lemma 2.1]{DZ94}.
\end{proof}

The next two lemmas are auxiliary results in order to understand the
definition of the set of admissible quadruples.

\begin{lemma}\label{lemma:q-cont}
Let $\Omega$ as in assumption \textbf{(A1)} and $j=1,2$.
\begin{enumerate}
\item[(i)] Let $\{z_k\}_{k\in \mathbb{N}}$ be a sequence in
$L^{\infty}(\Omega)$ with $z_k \in [m,M]$ a.e.  converging in the
$L^1(\Omega)$-norm to some element $z$ and $\{\psi_k^j\}_{k\in
\mathbb{N}}$ be a sequence in $\B$ converging in the $\bv$-norm to
some  $\psi^j\in \B$. Then $ q(z_k,\psi_k^1,\psi_k^2)$ converges to
$ q(z,\psi^1,\psi^2)$ in $L^1(\Omega)$.
\item[(ii)] Let $(z,\phi) \in  L^1(\Omega) \times H^1(\Omega)$, be such that
$H_\ve(\phi)\rightarrow z$ in $L^1(\Omega)$ as $\ve \to 0$ and let
$\psi^1,\psi^2\in \B$. Then $P_\ve(\phi,\psi^1,\psi^2) \rightarrow
q(z,\psi^1,\psi^2)$ in $L^1(\Omega)$ as $\ve \to 0$.
\item[(iii)] Given $\ve >0$, let $\{\phi_k\}_{k\in \N}$ be a sequence
in $H^1(\Omega)$ converging to $\phi \in H^1(\Omega)$ in the
$L^2$-norm. Then $H_\ve(\phi_k) \to H_\ve(\phi)$ in $L^1(\Omega)$,
as $k \to \infty$. Moreover, if $\{\psi_k^j\}_{k\in\mathbb{N}}$ are
sequences in $\B$, converging to some $\psi^j$ in $\B$, with respect
to the $L^1(\Omega)$-norm, then $q(H_\ve(\phi_k),\psi^1_k,\psi^2_k)
\to q(H_\ve(\phi),\psi^1,\psi^2)$ in $L^1(\Omega)$, as $k \to
\infty$.
\end{enumerate}
\end{lemma}
\begin{proof}
Since $\Omega$ is assumed to be bounded, we have $L^{\infty}(\Omega)
\subset L^{1}(\Omega)$ and $BV(\Omega)$ is continuous embedding in
$L^2(\Omega)$ \cite{EG92}. To prove {\em (i)}, notice that
\begin{align*}
\| q(z_k,\psi^1_k,\psi^2_k)   - q(z,\psi^1,\psi^2)\|_{L^1(\Omega)}
    &     = \norm{\psi^1_k z_k + \psi^2_k(1-z_k) - \psi^1z -
              \psi^2(1-z)}_{L^1(\Omega)}\\
   &
          \leq  \norm{z_k}_{L^\infty(\Omega)}
               \norm{\psi^1_k-\psi^1}_{L^1(\Omega)} + \norm{\psi^1}_{L^2(\Omega)}
               \norm{z_k-z}_{L^2(\Omega)} \\
 & \, + \norm{1-z_k}_{L^\infty(\Omega)}
 \norm{\psi^2_k-\psi^2}_{L^1(\Omega)}+ \norm{\psi^2}_{L^2(\Omega)}
 \norm{z_k-z}_{L^2(\Omega)}  \stackrel{k\to\infty}{\longrightarrow} 0
 \,.
\end{align*}
Here we use Lemma~\ref{lemma:0} in order to guarantee the
convergence of $z_k$ to $z$ in $L^2(\Omega)$.

Assertion {\em (ii)} follows with similar arguments and the fact
that $H_\ve(\phi) \in L^\infty(\Omega)$ for all $\ve >0$.

As  $ \| H_\ve(\phi_k) - H_\ve(\phi) \|_{L^1(\Omega)} \leq \ve^{-1}
\sqrt{ {\rm meas}(\Omega) } \| \phi_k - \phi \|_{L^2(\Omega)}$ the
first part of assertion {\em (iii)} follows. The second part of the
assertion {\em (iii)} holds  by a combination of the inequality
above and inequalities in the proof of assertion {\em (i)}.
\end{proof}

\begin{lemma} \label{lemma:auxil}
Let $\{\psi^j_k\}_{k\in \mathbb{N}}$ be a sequence of functions
satisfying Definition~\ref{def:quadruple} converging in $
L^1(\Omega)$ to some $\psi^j$, for $j=1,2$. Then $\psi^j$ also
satisfies Definition~\ref{def:quadruple}.
\end{lemma}
\noindent {\it Sketch of the proof.}

Let $k\in \N$ and $j=1,2$. Since $\psi^j_k$ satisfies
Definition~\ref{def:quadruple}, $\psi_k^j \in \bv$. From
\cite[Theorem 2, pg 172]{EG92} there exist sequences
$\{\psi_{k,l}^j\}_{l\in\N}$ in $\bv\times C^\infty(\Omega)$  such
that
$$
\psi_{k,l}^j \stackrel{l\to\infty}{\longrightarrow} \psi_k^j \,\,
\mbox{ in }\, \, L^1(\Omega) \,\, \mbox{ and } \, \,
|\psi_{k,l}^j|_\bv \stackrel{l\to\infty}{\longrightarrow}
|\psi_k^j|_\bv\; .
$$
In particular, for the subsequence $\{\psi_{k,l(k)}^j\}_{k\in\N}$
follows that
\begin{align}\label{eq:limit}
 \psi_{k,l(k)}^j \stackrel{k\to\infty}{\longrightarrow} \psi^j \,\, \mbox{ in }\,
\, L^1(\Omega) \,\, \mbox{ and } \, \, |\psi_{k,l(k)}^j|_\bv
\stackrel{k\to\infty}{\longrightarrow} |\psi^j|_\bv\; .
\end{align}
Moreover, by assumption $\psi^j \in L^1(\Omega)$. From the lower
semi-continuity of variational measure (see \cite[Theorem 1 pg.
  172]{EG92}), equation~\eqref{eq:limit} and the definition of $\bv$
space, it follows that $\psi^j \in \bv$. \mbox{}\hfill $\square$

In the next lemma we prove that the set of admissible quadruples is
closed with respect the $ L^1(\Omega)  \times  L^2(\Omega) \times
(L^1(\Omega))^2$ topology.

\begin{lemma} \label{lemma:limit-ad}
Let $(z_k,\phi_k,\psi^1_k,\psi^2_k)$ be a sequence of admissible
quadruples converging in $ L^1(\Omega)  \times  L^2(\Omega) \times
(L^1(\Omega))^2$ to some $(z,\phi,\psi^1,\psi^2)$, with $\phi \in
H^1(\Omega)$. Then, $(z,\phi,\psi^1,\psi^2)$ is also an admissible
quadruple.
\end{lemma}
\noindent {\it Sketch of the proof.}
Let $k\in \N$. Since $(z^1_k,\phi^1_k,\psi^1_k,\psi^2_k)$ is an
admissible quadruple, it follows from Definition~\ref{def:quadruple}
that there exist sequences $\{\phi_{k,l}\}_{l\in\N}$, in
$H^1(\Omega)$, $\{\psi_{k,l}^1\}_{l\in\N}$,
$\{\psi_{k,l}^2\}_{l\in\N}$ in $\bv\times C^\infty(\Omega)$
 and a correspondent sequence $\{\ve^l_{k}\}_{l\in\N}$ converging to zero such that
$$
\phi_{k,l} \stackrel{l\to\infty}{\longrightarrow} \phi_k \ \mbox{ in
} \ L^2(\Omega)\,, \quad H_{\ve^l_{k}}(\phi_{k,l})
\stackrel{l\to\infty}{\longrightarrow} z_k \ \mbox{ in } \
L^1(\Omega)\, \quad \mbox{ and } \quad |\psi_{k,l}^j|_\bv
\stackrel{l\to\infty}{\longrightarrow} |\psi_k^j|_\bv\,, j=1,2\, .
$$

 Define the monotone increasing function $\tau : \N \to \N$ such
that, for every $k \in \N$ it holds
\begin{equation} \label{eq:diag-argument}
\ve_k^{\tau(k)} \leq \frac{1}{2} \ve_{k-1}^{\tau(k-1)} \, , \ \big\|
\phi_{k,\tau(k)} - \phi_k \big\|_{ L^2(\Omega)} \leq \frac{1}{k} \,
, \ \big\| H_{\ve_k^{\tau(k)}}(\phi_{k,\tau(k)}) - z_k
\big\|_{L^1(\Omega)} \leq \frac{1}{k} \, , \,
|\psi_{k,\tau(k)}^j|_\bv \longrightarrow |\psi_k^j|_\bv\,, j=1,2\, .
\end{equation}
Hence, for each $k\in \N$
\begin{eqnarray*}
\big\| \phi - \phi_{k,\tau(k)} \big\|_{ L^2(\Omega)} & \leq &
    \| \phi - \phi_k \|_{L^2(\Omega)}
   + \big\| \phi_{k,\tau(k)}- \phi_k  \big\|_{L^2(\Omega)} \\
\big\| z - H_{\ve_k^{\tau(k)}}(\phi_{k,\tau(k)})
\big\|_{L^1(\Omega)} & \leq &
    \| z - z_k\|_{L^1(\Omega)}
   + \big\| H_{\ve_k^{\tau(k)}}(\phi_{k,\tau(k)}) - z_k
   \big\|_{L^1(\Omega)}\,.
\end{eqnarray*}
From \eqref{eq:diag-argument},
\begin{equation} \label{eq:zj-conv}
\lim_{k\to\infty} \big\| \phi - \phi_{k,\tau(k)} \big\|_{
L^2(\Omega)} = 0 \,\, , \,\,\lim_{k\to\infty} \big\| z -
H_{\ve_k^{\tau(k)}}(\phi_{k,\tau(k)})
                \big\|_{L^1(\Omega)} = 0\,.
\end{equation}
Moreover, with the same arguments as Lemma~\ref{lemma:auxil},
it follows that
$$|\psi_{k,\tau(k)}^j|_\bv \to |\psi^j|_\bv\,,\qquad j=1,2\,,$$
and $\psi^j\in \bv(\Omega)$.  Therefore, it remains to prove that
$(z,\phi,\psi^1,\psi^2)$ is an admissible quadruple. From
Definition~\ref{def:quadruple} and Lemma~\ref{lemma:auxil}, it is
enough to prove that $z \in L^\infty(\Omega)$.  If this is not the
case, there would exist a $\Omega' \subset \Omega$ with $|\Omega'| >
0$ and $\gamma > 0$ such that $z(x) > 1 + \gamma$ in $\Omega'$ (the
other case: $z(x) < - \gamma$ is analogous). Since
$(H_{\ve_k^{\tau(k)}}(\phi_{k,\tau(k)}))(x) \in [0,1]$ a.e. in
$\Omega$ for $k \in \N$ (see remark after
Definition~\ref{def:quadruple}), we would have
$$
\| z - H_{\ve_k^{\tau(k)}}(\phi_{k,\tau(k)}) \|_{L^1(\Omega)} \ge \|
z - H_{\ve_k^{\tau(k)}}(\phi_{k,\tau(k)}) \|_{L^1(\Omega')} \ge
   \gamma |\Omega'| \, , \ k \in \N \, ,
$$
contradicting the second limit in \eqref{eq:zj-conv}. \mbox{} \hfill
$\square$

\subsection{Relevant Properties of the Penalization
Functional}\label{subsec:relevant-propoerties-penalization}

In next lemmas, we verify properties of the functional $R$ which are
fundamental for the convergence analysis outlined in
Section~\ref{sec:conv-an}. In particular, these properties implies
that the level sets of $\Ga$ are compact in the set of admissible
quadruple, i.e., $\Ga$ assume a minimizer on this set. First, we prove
a lemma that simplify the functional $R$ in \eqref{def:R}. Here we
present the sketch of the proof. For more details, see the arguments
in \cite[Lemma 3]{DCLT09}.

\begin{lemma}\label{lemma:R-auxl}
Let $(z,\phi,\psi^1,\psi^2)$ be an admissible quadruple. Then, there
exists sequences $\{\ve_k\}_{k\in \N}$,  $\{\phi_k\}_{k\in \N}$ and
$\{\psi^j_k\}_{k\in \N}$ as in the Definition~\ref{def:quadruple},
such that
\begin{align}
R(z,\phi,\psi^1,\psi^2) = \lim_{k \to \infty} \left\{\beta_1
|H_{\ve_k}(\phi_k)|_\bv + \beta_2\norm{\phi_k -
\phi_0}^2_{H^1(\Omega)} + \beta_3
\sum_{j=1}^2|\psi^j_k-\psi^j_0|_\bv \right\}\,.
\end{align}
\end{lemma}
\noindent {\it Sketch of the proof.} For each $l \in \N$, the
definition of $R$ (see Definition~\ref{def:quadruple}) guaranties
the existence of sequences $\ve_k^l$, $\{\phi_{k,l}^j\}\in
H^1(\Omega)$ and $\{\psi_{k,l}^j\}\in \B$ such that
\begin{align*}
R(z,\phi,\psi^1,\psi^2) = \lim_{l\to \infty}\left\{ \liminf_{k\to
\infty}\left\{ \beta_1 |H_{\ve_k^l}(\phi_{k,l})|_\bv +
\beta_2\norm{\phi_{k,l} - \phi_0}^2_{H^1(\Omega)} \right\} +\beta_3
\sum_{j=1}^2|\psi^j_{k,l}-\psi^j_0|_\bv \right\}\,.
\end{align*}
Now a similar extraction of subsequences as in
Lemma~\ref{lemma:limit-ad} complete the proof. \mbox{} \hfill
$\square$

In the following, we prove two lemmas that are essential to the
proof of  well posedness of the Tikhonov functional \eqref{eq:mc}.

\begin{lemma} \label{lemma:rho-coer}
The functional $R$ in \eqref{def:R} is coercive on the set of
admissible quadruples. In other words, given any admissible
quadruple $(z, \phi, \psi^1, \psi^2)$ we have
\begin{align*}
R(z, \phi, \psi^1, \psi^2) \geq \left(\beta_1|z|_{\bv}
+\beta_2\norm{\phi - \phi_0}^2_{H^1(\Omega)} +\beta_2 \sum_{j=1}^2
|\psi^j - \psi_0^j|_{\bv}\right)\,.
\end{align*}
\end{lemma}
\noindent {\it Sketch of the proof.} Let $(z,\phi,\psi^1,\psi^2)$ be
an admissible quadruple. From \cite[Lemma~4]{CLT08}, it  follows
that
\begin{align} \label{eq:rho-coercivo}
\rho(z,\phi) \ \geq \
\big( \beta_1 |z|_\bv + \beta_2 \|\phi-\phi_0\|_{H^1(\Omega)}^2 \big) \, .
\end{align}
Now, from \eqref{eq:rho-coercivo} and the definition of $R$ in
\eqref{def:R}, we have
$$
\big( \beta_1 |z|_\bv + \beta_2 \| \phi - \phi_0 \|^2_{H^1(\Omega)}
         + \beta_3 \sum_{j=1}^2|\psi^j-\psi_0^j|_\bv\big)
\, \leq \, \rho(z,\phi) + \beta_3 \sum_{j=1}^2|\psi^j-\psi_0^j|_\bv
\, = \, R(z,\phi,\psi^1,\psi^2) \, ,
$$
concluding the proof. \mbox{} \hfill $\square$

\begin{lemma} \label{lemma:rho-lsc}
The functional $R$ in \eqref{def:R} is weak lower semi-continuous on
the set of admissible quadruples, i.e. given a sequence
$\{(z_k,\phi_k,\psi^1_k,\psi^2_k)\}$ of admissible quadruples such
that $z_k \to z$ in $L^1(\Omega)$, $\phi_k \rightharpoonup \phi$ in
$H^1(\Omega)$, $\psi^j_k \to \psi^j$ in $L^1(\Omega)$, for some
admissible quadruple $(z,\phi,\psi^1,\psi^2)$, then
$$
R(z,\phi,\psi^1,\psi^2) \ \leq \ \liminf_{k \in \N}
R(z_k,\phi_k,\psi^1_k,\psi^2_k) \, .
$$
\end{lemma}
\begin{proof}
The functional $\rho(z,\phi)$ is weak lower semi-continuous cf.
\cite[Lemma 5]{CLT08}. As $\psi_k^j \in \bv$ follows from
\cite[Theorem 2 pg 172]{EG92} that there exist sequences
$\{\psi_{k,l}^j\} \in \bv \cap C^\infty(\Omega)$ such that
$\norm{\psi_{k,l}^j -  \psi_k^j}_{L^1(\Omega)} \leq \frac{1}{l}$.
From a diagonal argument, we can extract a subsequence
$\{\psi_{k,l(k)}^j\}$ of $\{\psi_{k,l}^j\}$ such that
$\{\psi_{k,l(k)}^j\} \to \psi^j$ in $L^1(\Omega)$ as $k \to \infty$.
Let $\xi \in C^1_c(\Omega, \mathbb{R}^n)\,,\,|\xi|\leq 1$. Then,
from \cite[Theorem 1 pg 167]{EG92}, it follows that
\begin{align*}
\int_\Omega \psi^j\; \nabla \cdot \xi dx & = \lim_{k\to
\infty}\int_\Omega
                                    \psi_{k,l(k)}^j \; \nabla \cdot \xi dx =
                                    \lim_{k\to \infty}\left[\int_\Omega \left(\psi_{k,l(k)}^j -
                                    \psi_k^j\right) \; \nabla \cdot \xi dx + \int_\Omega \psi_k^j \; \nabla \cdot \xi dx
                                    \right]\\
                                  &
                                  \leq
                               \lim_{k\to \infty}\left[\norm{\psi_{k,l(k)}^j -
                               \psi_k^j}_{L^1(\Omega)} \norm{\nabla \cdot \xi}_{L^\infty(\Omega)}|\Omega| -
                               \int_\Omega \xi\cdot \sigma_k d|\psi_k^j|_\bv  \right]
                           \leq
                                \liminf_{k\to \infty} |\psi_k^j|_\bv\,.
\end{align*}
Thus, form the definition of $|\cdot|_\bv$ (see \cite{EG92}), we
have
\begin{align*}
|\psi^j|_\bv = \sup\left\{ \int_\Omega \psi^j\; \nabla \cdot \xi
dx\,; \xi \in C^1_c(\Omega, \mathbb{R}^n)\,,\,|\xi|\leq 1 \right\}
\leq \liminf_{k\to \infty} |\psi_k^j|_\bv\,.
\end{align*}
Now, the lemma follows from the fact that the functional $R$ in
\eqref{def:R} is a linear combination of lower semi-continuous
functionals.
\end{proof}

\section{Convergence Analysis} \label{sec:conv-an}

In the following, we consider any positive parameter $\alpha$,
$\beta_j\,, j=1,2,3$ as in the  general assumption to this article.
First, we prove that  the functional $\Ga$ in \eqref{eq:gzphi} is
well posed.

\begin{theo}[{\rm Well-Posedness}] \label{th:admissible}
The functional $\Ga$ in \eqref{eq:gzphi} attains minimizers on the
set of admissible quadruples.
\end{theo}
\begin{proof}
Notice that, the set of admissible quadruples is not empty, since
$(0,0,0,0)$ is admissible. Let $\{(z_k,\phi_k,\psi^1_k,\psi^2_k) \}$
be a minimizing sequence for $\Ga$, i.e. a sequence of admissible
quadruples satisfying ${\cal{G}}_\alpha(z_k,\phi_k,
\psi^1_k,\psi^2_k) \to \inf {\cal{G}}_\alpha \leq
{\cal{G}}_\alpha(0,0,0,0) < \infty$. Then, $\{ \Ga(z_k,\phi_k,
\psi^1_k,\psi^2_k) \}$ is a bounded sequence of real numbers.
Therefore, $\{(z_k,\phi_k,\psi^1_k,\psi^2_k)\}$ is uniformly bounded
in $\bv \times H^1(\Omega)\times \bv^2$. Thus, from the Sobolev
Embedding Theorem \cite{Ada75, EG92}, we guarantee the existence of
a subsequence (denoted again by $\{ (z_k, \phi_k,\psi^1_k,\psi^2_k)
\}$) and the existence of $(z,\phi,\psi^1,\psi^2) \in L^1(\Omega)
\times H^1(\Omega)\times \bv^2$ such that $\phi_k \to \phi$ in
$L^2(\Omega)$, $\phi_k \rightharpoonup \phi$ in $H^1(\Omega)$, $z_k
\to z$ in $L^1(\Omega)$ and $\psi_k^j \to \psi^j \mbox{ in }
L^1(\Omega)$. Moreover, $z, \psi^1 \mbox{ and } \psi^2 \in \bv$. See
\cite[Theorem 4, pp. 176]{EG92}.

From Lemma~\ref{lemma:limit-ad}, we conclude that
$(z,\phi,\psi^1,\psi^2)$ is an admissible quadruple. Moreover, from
the weak lower semi-continuity of $R$ (Lemma~\ref{lemma:rho-lsc}),
together with the continuity of  $q$ (Lemma~\ref{lemma:q-cont}) and
continuity of $F$ (see the general assumption),  we obtain
\begin{eqnarray}
\inf \ \Ga & = & \lim_{k \to \infty} {\cal{G}}_\alpha(z_k,
\phi_k,\psi^1_k,\psi^2_k) \ = \ \lim_{k \to \infty} \big\{
\|F(q(z_k,\psi^1_k,\psi^2_k))-y^\delta\|^2_Y
      + \alpha R(z_k,\phi_k,\psi^1_k,\psi^2_k) \big \} \nonumber \\
& \geq & \| F(q(z,\psi^1,\psi^2)) - y^\delta \|^2_Y + \alpha
R(z,\phi,\psi^1,\psi^2) \   =  \ \Ga(z,\phi,\psi^1,\psi^2) \, ,
\label{eq:no}
\end{eqnarray}
proving that $(z,\phi,\psi^1,\psi^2)$ minimizes $\Ga$.
\end{proof}

In that follows, we shall denote a minimizer of $\Ga$ by $(z_\alpha,
\phi_\alpha,\psi^1_\alpha,\psi^2_\alpha)$. In particular the
functional $\hat{\Ga}$ in \eqref{eq:m-reg} attain a generalized
minimizer in the sense of Definition~\ref{def:quadruple}. In the
next theorem, we summarize some convergence results for the
regularized minimizers. These results are based on the existence of
a generalized \emph{minimum norm solutions}.

\begin{df}\label{def:R-min-norm-solution}
An admissible quadruple   $(z^{\dag}, \phi^{\dag}, \psi^{1,\dag},
\psi^{2,\dag})$ is called a \textit{$R$-minimizing solution} if
satisfies
\begin{align*}
(i) \, & \, F( q(z^{\dag},\psi^{1,\dag},\psi^{2,\dag}) ) = y \, , \\
(ii)\, & \, R(z^{\dag}, \phi^{\dag},\psi^{1,\dag},\psi^{2,\dag}) =
\mbox{ms}
         := \inf \big\{ R(z, \phi,\psi^1,\psi^2) \, ;\
         (z, \phi,\psi^1,\psi^2) \ \mbox{\rm is an} \\
       & \hskip5cm \mbox{\rm admissible quadruple and} \
         F(q(z,\psi^1,\psi^2)) = y \big\} \; .
\end{align*}
\end{df}

\begin{theo}[{\rm $R$-minimizing solutions}] \label{th:min-norm}
Under the general assumptions of this paper there exists a
$R$-minimizing solution.
\end{theo}
\begin{proof}
From the general assumption on this paper  and
Remark~\ref{remark-ad}, we conclude that the set of admissible
quadruple satisfying $F(q(z,\psi^1,\psi^2)) = y$  is not empty.
Thus, $ms$ in ({\it ii}) is finite and there exists a sequence
$\{(z_k,\phi_k,\psi^1_k,\psi^2_k)\}_{k\in\N}$ of admissible
quadruple satisfying
$$
F(q(z_k,\psi^1_k,\psi^2_k)) = y \quad \mbox{and} \quad
R(z_k,\phi_k,\psi^1_k,\psi^2_k) \to ms < \infty \; .
$$
Now, form the definition of $R$, it follows that the sequences $\{
\phi_k \}_{k\in\N}$, $\{ z_k \}_{k\in\N}$ and $\{ \psi^j_k
\}_{k\in\N}^{j=1,2}$ are uniformly bounded in $H^1(\Omega)$ and
$\bv(\Omega)$,  respectively. Then, from the Sobolev Compact
Embedding Theorem \cite{Ada75, EG92}, we have (up to subsequences)
that
$$
\phi_k \to \phi^{\dag} \ \mbox{ in } \ L^2(\Omega)\,, \quad  z_k \to
z^{\dag} \ \mbox{ in } \ L^1(\Omega) \quad \mbox{ and } \quad
\psi^j_k \to \psi^{j,\dag} \ \mbox{ in } \ L^1(\Omega)  \, , \ j =
1,2 \; .
$$
Lemma~\ref{lemma:limit-ad} implies that $(z^{\dag},
\phi^{\dag},\psi^{1,\dag},\psi^{2,\dag})$ is an admissible
quadruple. Since $R$ is weakly lower semi-continuous (cf. Lemma
\ref{lemma:rho-lsc}), it follows
$$
ms = \liminf_{k\to\infty} R(z_k,\phi_k,\psi^1_k,\psi^2_k)
     \geq R(z^{\dag},\phi^{\dag},\psi^{1,\dag},\psi^{2,\dag}) \; .
$$
Moreover, we conclude from Lemma~\ref{lemma:q-cont} that
$$ q(z^{\dag},\psi^{1,\dag},\psi^{1,\dag}) = \lim\limits_{k\to\infty}
q(z_k,\psi^1_k,\psi^2_k) \quad \mbox{ and } \quad F(
q(z^{\dag},\psi^{1,\dag},\psi^{2,\dag}) ) = \lim\limits_{k\to\infty}
F(q(z_k,\psi^1_k,\psi^2_k)) = y\,.$$ Thus,
$(z^{\dag},\phi^{\dag},\psi^{1,\dag},\psi^{2,\dag})$ is a $R$-
minimizing solution.
\end{proof}

Using classical techniques from the analysis of Tikhonov
regularization methods (see \cite{EngKunNeu89, EngHanNeu96}), we
present below the main convergence and stability theorems of this
paper. The arguments in the proof are somewhat different of that
presented in \cite{DCLT09, DCLT2010}. But, for sake of completeness,
we present the proof.

\begin{theo}[{\bf Convergence for exact data}] \label{th:converg}
Assume that we have exact data, i.e. $y^\delta=y$. For every $\alpha
> 0$ let $(z_\alpha,\phi_\alpha,\psi^1_\alpha,\psi^2_\alpha)$ denote a
minimizer of $\Ga$ on the set of admissible quadruples. Then, for
every sequence of positive numbers $\{\alpha_k\}_{k\in\N}$
converging to zero there exists a subsequence, denoted again by
$\{\alpha_k\}_{l\in\N}$, such that
$(z_{\alpha_k},\phi_{\alpha_k},\psi^1_{\alpha_k},\psi^2_{\alpha_k})$
is strongly convergent in $L^1(\Omega) \times L^2(\Omega)\times
(L^1(\Omega))^2$. Moreover, the limit is a solution of
\eqref{eq:inv-probl}.
\end{theo}
\begin{proof}
Let $(z^{\dag}, \phi^{\dag},\psi^{1,\dag},\psi^{2,\dag})$ be a
$R$-minimizing solution of \eqref{eq:inv-probl} -- its existence is
guaranteed by Theorem~\ref{th:min-norm}. Let $\{\alpha_k\}_{k\in\N}$
be a sequence of positive numbers converging to zero. For each $k
\in \N$, denote $(z_k,\phi_k,\psi^1_k,\psi^2_k) :=
(z_{\alpha_k},\phi_{\alpha_k}, \psi^1_{\alpha_k},\psi^2_{\alpha_k})$
be a minimizer of $G_{\alpha_k}$. Then, for each $k \in \N$, we have
\begin{equation} \label{eq:0}
G_{\alpha_k}(z_k,\phi_k,\psi^1_k,\psi^2_k)
 \leq \big\| F(q(z^{\dag},\psi^{1,\dag},\psi^{2,\dag})) - y \big\|
 + \alpha_k R(z^{\dag}, \phi^{\dag},\psi^{1,\dag},\psi^{2,\dag})
 = \alpha_k R(z^{\dag}, \phi^{\dag},\psi^{1,\dag},\psi^{2,\dag}) .
\end{equation}
Since $\alpha_k R(z_k,,\phi_k,\psi^1_k,\psi^2_k) \le
G_{\alpha_k}(z_k, \phi_k,\psi^1_k,\psi^2_k)$, it follows from
\eqref{eq:0} that
\begin{equation} \label{eq:1}
R(z_k,\phi_k,\psi^1_k,\psi^2_k) \ \leq \
R(z^{\dag},\phi^{\dag},\psi^{1,\dag},\psi^{2,\dag}) \ < \ \infty  \;
.
\end{equation}
Moreover, from the assumption on the sequence $\{ \alpha_k \}$, it follows that
\begin{equation} \label{eq:2}
\lim_{k\to\infty} \ \alpha_k R
        (z^{\dag}, \phi^{\dag},\psi^{1,\dag},\psi^{1,\dag}) \ = \ 0 \; .
\end{equation}
From \eqref{eq:1} and Lemma~\ref{lemma:rho-coer}, we conclude that
sequences $\{ \phi_k \}$, $\{ z_k \}$ and $\{\psi^j_k\}$ are bounded
in $H^1(\Omega)$\,and  $\bv$, respectively, for $j=1,2$.
Using an argument of extraction of diagonal subsequences (see proof of Lemma~%
\ref{lemma:limit-ad}), we can guarantee the existence of an
admissible quadruple $(\tilde{z},\tilde{\phi},
\tilde{\psi}^1,\tilde{\psi}^2)$ such that
$$
(z_k, \phi_k,\psi^1_k,\psi^2_k) \to (\tilde{z},\tilde{\phi},
\tilde{\psi}^1,\tilde{\psi}^2) \ \mbox{ in } \ L^1(\Omega) \times
L^2(\Omega)\times (L^1(\Omega))^2 \, .
$$
Now, from Lemma~\ref{lemma:q-cont}~(i), it follows that
$q(\tilde{z},\tilde{\psi}^1,\tilde{\psi}^2) =
\lim\limits_{k\to\infty} q(z_k,\psi^1_k,\psi^2_k)$ in $L^1(\Omega)$.
Using the continuity of the operator $F$ together with \eqref{eq:0}
and \eqref{eq:2}, we conclude that
$$
y \ = \ \lim_{k\to\infty} F( q(z_k,\psi^1_k,\psi^2_k)) \ = \
F(q(\tilde{z},\tilde{\psi}^1,\tilde{\psi}^2)) \; .
$$
On the other hand, from the lower semi-continuity of $R$ and
\eqref{eq:1} it follows that
$$
R(\tilde{z}, \tilde{\phi},\tilde{\psi}^1,\tilde{\psi}^2) \ \leq \
\liminf_{k\to\infty} R(z_k,\phi_k,\psi^1_k,\psi^2_k) \ \leq \
\limsup_{k\to\infty} R(z_k,\phi_k,\psi^1_k,\psi^2_k)) \ \leq \
R(z^{\dag},\phi^{\dag},\tilde{\psi}^1,\tilde{\psi}^2) \, ,
$$
concluding the proof.
\end{proof}

\begin{theo}[{\bf Stability}] \label{th:stabil}
Let $\alpha = \alpha(\delta)$ be a function satisfying
$\lim\limits_{\delta \to 0} \alpha (\delta)$ $= 0$ and
$\lim\limits_{\delta \to 0} \delta^2 \alpha(\delta)^{-1} = 0$.
Moreover, let $\{ \delta_k \}_{k\in\N}$ be a sequence of positive
numbers converging to zero and $y^{\delta_k} \in Y$ be corresponding
noisy data satisfying \eqref{eq:noisy-data}. Then, there exist a
subsequence, denoted again by $\{ \delta_k \}$, and a sequence $\{
\alpha_k := \alpha(\delta_k) \}$ such that
$(z_{\alpha_k},\phi_{\alpha_k},
\psi^1_{\alpha_k},\psi^2_{\alpha_k})$ converges in $L^1(\Omega)
\times L^2(\Omega)\times (L^1(\Omega))^2$ to solution of
\eqref{eq:inv-probl}.
\end{theo}
\begin{proof}
Let $(z^{\dag},\phi^{\dag}, \psi^{1,\dag},\psi^{1,\dag})$ be a
$R$-minimizer solution of \eqref{eq:inv-probl} (such existence is
guaranteed by Theorem~\ref{th:min-norm}). For each $k \in \N$, let
$(z_k,\phi_k, \psi^1_k,\psi^2_k) :=
(z_{\alpha(\delta_k)},\phi_{\alpha(\delta_k)},
\psi^1_{\alpha(\delta_k)},\psi^2_{\alpha(\delta_k)})$ be a minimizer
of $G_{\alpha(\delta_k)}$. Then, for each $k \in \N$ we have
\begin{align} \label{eq:0a}
G_{\alpha_k} (z_k,\phi_k,\psi^1_k,\psi^2_k)
 & \leq \big\| F( q(z^{\dag},\psi^{1,\dag},\psi^{1,\dag}) ) - y^{\delta_k} \big\|_Y^2
        + \alpha(\delta_k) R(z^{\dag},\phi^{\dag},\psi^{1,\dag},\psi^{2,\dag})
        \nonumber \\
 & \leq \ \delta_k^2 + \alpha(\delta_k)
        R(z^{\dag},\phi^{\dag}, \psi^{1,\dag},\psi^{2,\dag}) \, .
\end{align}
From \eqref{eq:0a} and the definition of $G_{\alpha_k}$, it follows
that
\begin{equation} \label{eq:1a}
R(z_k,\phi_k,\psi^1_k,\psi^2_k) \ \leq \
  \frac{\delta_k^2}{\alpha(\delta_k)} +
  R(z^{\dag}, \phi^{\dag}, \psi^{1,\dag},\psi^{2,\dag}) \; .
\end{equation}

Taking the limit as $k\to\infty$ in \eqref{eq:1a}, it follows from
theorem assumptions on $\alpha(\delta_k)$, that
$$
\lim_{k\to\infty} \norm{ F( q(z_k,\psi^1_k,\psi^2_k)) - y^{\delta_k}
} \ \leq \ \lim_{k\to\infty} \left(\delta_k^2 + \alpha(\delta_k)
        R(z^{\dag},\phi^{\dag}, \psi^{1,\dag},\psi^{2,\dag}) \right) \ = \ 0 \,
        ,
$$
and
\begin{equation} \label{eq:1b}
\limsup_{k\to\infty} R(z_k,\phi_k,\psi^1_k,\psi^2_k) \ \leq \
R(z^{\dag},\phi^{\dag},\psi^{1,\dag},\psi^{2,\dag}) \; .
\end{equation}
With the same  arguments as in the proof of
Theorem~\ref{th:converg}, we conclude that, at least a subsequence
that we denote again by $(z_k,\phi_k,\psi^1_k,\psi^2_k)$, converge
in $L^1(\Omega)\times L^2(\Omega) \times (L^1(\Omega))^2$ to some
admissible quadruple $(z,\phi,\psi^1,\psi^2)$. Moreover, by taking
the limit as $k\to\infty$ in \eqref{eq:0a}, it follows from the
assumption on $F$ and Lemma~\ref{lemma:q-cont} that
$$F(q(z,\phi,\psi^1,\psi^2))=\lim\limits_{k\to\infty}
F(q(z_k,\psi^1_k,\psi^2_k)) = y\,.$$
\end{proof}

The functional $\Ga$ defined in \eqref{eq:gzphi} is not easy to
handled numerically, i.e., we are not able to derive a suitable
optimality condition to the minimizers of $\Ga$. In the next
section, we work in sight to surpass such difficulty.

\section{Numerical Solution} \label{sec:num-sol}

In this section,  we introduce a functional which can be handled
numerically, and whose minimizers are 'near' to the minimizers of
$\Ga$. Let $\Ger$ be the  functional defined by
\begin{equation} \label{eq:m-reg}
\Ger(\phi,\psi^1,\psi^2) := \| F(P_\ve(\phi,\psi^1,\psi^2)) -
y^\delta \|_Y^2 + \alpha \big( \beta_1 |H_\ve(\phi)|_{\bv} + \beta_2
\| \phi - \phi_0 \|_{H^1}^2 + \beta_3 {\textstyle
\sum\limits_{j=1}^2|\psi^j-\psi^j_0|_\bv} \big) ,
\end{equation}
where $P_\ve(\phi,\psi^1,\psi^2) := q(H_\ve(\phi), \psi^1,\psi^2)$
is defined in \eqref{eq:def-Pve}. The functional $\Ger$ is
well-posed as the following lemma shows:

\begin{lemma}\label{lemma:Gae}
Given positive constants $\alpha$, $\ve$, $\beta_j$ as in the
general assumption of this article, $\phi_0 \in H^1(\Omega)$ and
$\psi_0^j \in \B$, $j = 1,2$. Then, the functional $\Ger$ in
\eqref{eq:m-reg} attains a minimizer on $H^1(\Omega) \times
(\bv)^2$.
\end{lemma}
\begin{proof}
Since, $\inf \{\Ger(\phi,\psi^1,\psi^2) \, : \ (\phi,\psi^1,\psi^2)
\in H^1(\Omega) \times (\bv)^2 \} \leq \Ger(0,0,0) < \infty$, there
exists a minimizing sequence $\{ (\phi_k,\psi^1_k,\psi^2_k) \}$ in
$H^1(\Omega)\times \B^2$ satisfying
$$
\lim_{k\to\infty} \Ger(\phi_k, \psi^1_k,\psi^2_k) \ = \ \inf \{
\Ger(\phi,\psi^1,\psi^2) \; : \ (\phi,\psi^1,\psi^2) \in H^1(\Omega)
\times
     \B^2 \} \, .
$$
Then, for fixed $\alpha > 0$, the definition of $\Ger$ in
\eqref{eq:m-reg} implies that the sequences $\{ \phi_k \}$ and $\{
\psi^j_k\}^{j=1,2}$ are bounded in $H^1(\Omega)$ and $(\bv)^2$,
respectively. Therefore, from Banach-Alaoglu-Bourbaki Theorem
\cite{Yosida95} $\phi_k \rightharpoonup \phi$ in $H^1(\Omega)$ and
from  \cite[Theorem 4 pg. 176]{EG92}, $\psi_k^j \rightarrow \psi^j$
in $L^1(\Omega)$, $j=1,2$. Now,  a similar argument as in
Lemma~\ref{lemma:auxil} implies that $\psi^j\in \B$, for $j=1,2$.
Moreover, by the weak lower semi-continuity of the $H^1$--norm
\cite{Yosida95} and $|\cdot|_\bv$ measure (see \cite[Theorem 1 pg.
172]{EG92}), it follows that $$\| \phi - \phi_0\|^2_{H^1} \leq
\liminf\limits_{k\to\infty} \| \phi_k - \phi_0 \|^2_{H^1} \, \mbox{
and }\; \ |\psi^j-\psi^j_0|_\bv \leq \liminf\limits_{k\to\infty} |
\psi^j_k - \psi^j_0 |_\bv\;.$$

The  compact embedding of $H^1(\Omega)$ into $L^2(\Omega)$
\cite{Ada75} implies in the existence of a subsequence of
$\{\phi_k\}$, (that we denote with the same index) such that $\phi_k
\to \phi$ in $L^2(\Omega)$. Follows from Lemma~\ref{lemma:q-cont}
and \cite[Theorem 1, pg 172]{EG92} that $| H_\ve(\phi) |_{\bv} \leq
\liminf\limits_{k\to\infty} | H_\ve(\phi_k) |_{\bv}$. Hence, from
continuity of $F$ in $L^1$, continuity of  $q$ (see
Lemma~\ref{lemma:q-cont}), together with the estimates above, we
conclude that
\begin{align*}
\Ger(\phi,\psi^1,\psi^2)
 & \leq \lim_{k\to\infty} \| F( P_\ve(\phi_k, \psi^1_k,\psi^2_k)) - y^\delta \|_{Y}^2 \\
 & \quad + \alpha \, \Big( \beta_1 \liminf_{k\to\infty} |H_\ve(\phi_k)|_{\bv}
   + \beta_2 \liminf_{k\to\infty} \|\phi_k - \phi_0\|_{H^1(\Omega)}^2
   + \beta_3 \liminf_{k\to\infty} {\textstyle\sum\limits_{j=1}^2 |\psi^j_k - \psi^j_0|_\bv}
     \Big) \\
 & \leq \liminf_{k\to\infty} \Ger(\phi_k, \psi^1_k,\psi^2_k) \ = \ \inf \Ger \, ,
\end{align*}
Therefore, $(\phi,\psi^1,\psi^2)$ is a minimizer of $\Ger$.
\end{proof}

In the sequel, we prove that, when $\ve \to 0$, the minimizers of
$\Ger$ approximate a minimizer of the functional $\Ga$. Hence,
numerically, the minimizer of $\Ger$ can be used as a suitable
approximation for the minimizers of $\Ga$.

\begin{theo} \label{th:just}
Let $\alpha$ and $\beta_j$ be given as in the general assumption of
this article. For each $\ve > 0$, denote by $(\phi_{\ve,\alpha},
\psi^1_{\ve,\alpha}, \psi^2_{\ve,\alpha})$ a minimizer of $\Ger$
(that there exist form Lemma~\ref{lemma:Gae}). Then, there exists a
sequence of positive numbers $\ve_k \to 0$ such that
$(H_{\ve_k}(\phi_{{\ve_k},\alpha}), \phi_{{\ve_k},\alpha},
\psi^1_{{\ve_k},\alpha}, \psi^2_{{\ve_k},\alpha})$ converges
strongly in $L^1(\Omega) \times \lzo \times (L^1(\Omega))^2$ and the
limit minimizes $\Ga$ on the set of admissible quadruples.
\end{theo}
\begin{proof}
Let $(z_\alpha,\phi_{\alpha}, \psi^1_{\alpha},\psi^2_{\alpha})$ be a
minimizer of the functional $\Ga$ on the set of admissible
quadruples (cf. Theorem~\ref{th:admissible}). From
Definition~\ref{def:quadruple}, there exists a sequence $\{ \ve_k
\}$ of positive numbers converging to zero and corresponding
sequences $\{ \phi_k \}$ in $H^1(\Omega)$ satisfying $\phi_k \to
\phi_{\alpha}$ in $L^2(\Omega)$, $H_{\ve_k}(\phi_k) \to z_\alpha$ in
$L^1(\Omega)$ and, finally, sequences $\{ \psi_k^j \}$ in $\bv\times
C_c^\infty(\Omega)$ such that $|\psi_k^j|_\bv \longrightarrow
|\psi^j|_\bv$ . Moreover, we can further assume (see
Lemma~\ref{lemma:R-auxl}) that
$$
R(z_\alpha, \phi_{\alpha},\psi^1_{\alpha},\psi^2_{\alpha}) \ = \
\lim_{k\to\infty}
  \big( \beta_1 |H_{\ve_k} (\phi_k)|_\bv + \beta_2 \|\phi_k -\phi_0\|^2_{H^1(\Omega)}
  + \beta_3 {\textstyle\sum\limits_{j=1}^2|\psi_k^j-\psi^j_0|_\bv} \big) \, .
$$
Let $(\phi_{\ve_k},\psi^1_{\ve_k},\psi^2_{\ve_k})$ be a minimizer of
${\cal{G}}_{\ve_k,\alpha}$. Hence,
$(\phi_{\ve_k},\psi^1_{\ve_k},\psi^2_{\ve_k})$ belongs to
$H^1(\Omega)\times \B^2$ (see Lemma~\ref{lemma:Gae}). The sequences
$\{ H_{\ve_k}(\phi_{\ve_k}) \}$,$\{ \phi_{\ve_k} \}$ and
$\{\psi_{\ve_k}^j\}$ are uniformly bounded in $\bv(\Omega)$,
$H^1(\Omega)$  and $\bv(\Omega)$, for $j=1,2$, respectively.
Form compact embedding (see Theorems \cite{Ada75} and \cite[Theorem
4 pg. 176]{EG92}),  there exist convergent subsequences whose limits
are denoted by $\tilde{z}$, $\tilde{\phi}$ and $\tilde{\psi}^j$
belong to $\bv(\Omega)$,$H^1(\Omega)$ and $\bv(\Omega)$, for
$j=1,2$, respectively.

Summarizing, we have $\phi_{\ve_k} \to \tilde{\phi}$ in
$L^2(\Omega)$, $H_{\ve_k}(\phi_{\ve_k}) \to \tilde{z}$ in
$L^1(\Omega)$, and $\psi^j_{\ve_k} \to \tilde{\psi}^j$ in
$L^1(\Omega)$, $j = 1,2$. Thus, $(\tilde{z},\tilde{\phi},
\tilde{\psi}^1,\tilde{\psi}^2) \in L^1(\Omega) \times H^1(\Omega)
\times \L^1(\Omega)$ is an admissible quadruple (cf.
Lemma~\ref{lemma:limit-ad}).

From the definition of $R$, Lemma~\ref{lemma:q-cont} and the continuity
of $F$, it follows that
\begin{eqnarray*}
& \| F(q(\tilde{z},\tilde{\psi}^1,\tilde{\psi}^2)) -  y^\delta
\|^2_Y \ = \
   \lim_{k\to\infty} \| F( P_{\ve_k}(\phi_{\ve_k},\psi^1_{\ve_k},\psi^2_{\ve_k}) )
                                                     - y^\delta \|_Y^2 \, , & \\
& R(\tilde{z},\tilde{\phi}, \tilde{\psi}^1,\tilde{\psi}^2) \ \leq \
  \liminf_{k\to\infty} \big( \beta_1 | H_{\ve_k}(\phi_{\ve_k}) |_\bv +
  \beta_2 \| \phi_{\ve_k}-\phi_0 \|^2_{H^1(\Omega)} +
  \beta_3 {\textstyle\sum_{j=1}^2 |\psi^j_{\ve_k} - \psi_0^j|_\bv} \Big) \, . &
\end{eqnarray*}
Therefore,
\begin{eqnarray*}
\Ga(\tilde{z}, \tilde{\phi},\tilde{\psi}^1,\tilde{\psi}^2)  & = &
  \| F(q(\tilde{z},\tilde{\psi}^1,\tilde{\psi}^2)) - y^\delta \|_Y^2
  + \alpha R(\tilde{z}, \tilde{\phi},\tilde{\psi}^1,\tilde{\psi}^2) \\
 & \le & \liminf_{k\to\infty} \ {\cal{G}}_{\ve_k,\alpha}
         (\phi_{\ve_k}, \psi^1_{\ve_k},\psi^2_{\ve_k})
 \ \le \ \liminf_{k\to\infty} {\cal{G}}_{\ve_k,\alpha}(\phi_k, \psi^1_k,\psi^2_k) \\
 & \le & \limsup_{k\to\infty}
         \| F( P_{\ve_k}(\phi_k, \psi^1_k,\psi^2_k) ) - y^\delta \|_Y^2 \\
 &     & + \ \alpha \limsup_{k\to\infty}  \big( \beta_1 | H_{\ve_k}(\phi_k) |_\bv
          + \beta_2 \| \phi_k - \phi_0 \|^2_{H^1(\Omega)}
          + \beta_3 {\textstyle\sum_{j=1}^2 |\psi^j_k-\psi^j_0|_\bv} \big) \\
 &  =  & \| F(q(z_\alpha, \psi^1_\alpha,\psi^2_\alpha)) - y^\delta \|_Y^2
         + \alpha R(z_\alpha, \phi_{\alpha},\psi^1_\alpha,\psi^2_\alpha)
 \  =  \ \Ga(z_\alpha, \phi^1_{\alpha},\psi^1_\alpha,\psi^2_\alpha) \, ,
\end{eqnarray*}
characterizing $(\tilde{z},
\tilde{\phi},\psi^1_\alpha,\psi^2_\alpha)$ as a minimizer of $\Ga$.
\end{proof}

%
\subsection{Optimality Conditions for the Stabilized Functional}

For numerical purposes it is convenient to derive first order
optimality conditions for minimizers of the functional $\Ga$. Since
$P$ is a discontinuous operator, it is not possible. However, thanks
to the Theorem~\ref{th:stabil}, the minimizers of the stabilized
functionals $\Ger$ can be used for approximate minimizers of the
functional $\Ga$. Therefore, we consider $\Ger$ in \eqref{eq:m-reg}
,with $Y$ a Hilbert space, and we look for the G\^ateaux directional
derivatives with respect to $\phi$ and the unknown $\psi^j$ for
$j=1,2$.

Since $H'_\ve(\phi)$ is self-adjoint\footnote{Note that
$H'_\ve(t)=\begin{cases} \frac{1}{\ve} \,\,
t\in (-\ve,0)\\
0\,\, \mbox{ other else}\,.
\end{cases}$}, we can write
the optimality conditions for the functional $\Ger$ in the form of
the system
\begin{subequations} \label{eq:formal0}
\begin{eqnarray}
& \alpha (\Delta-I)(\phi - \phi_0) \  = \
L_{\ve,\alpha,\beta}(\phi,\psi^1,\psi^2) \,,\qquad &  \ {\rm in } \
\Omega \, \\
& (\phi - \phi_0) \cdot \nu \  = \ 0 \, , \qquad\qquad \quad \quad   & \ {\rm at } \ \partial\Omega  \\
& \alpha \, \nabla \cdot \big[\nabla(\psi^j-\psi^j_0)/|\nabla
(\psi^j-\psi^j_0)|\big] \ = \
L^j_{\ve,\alpha,\beta}(\phi,\psi^1,\psi^2) \, , & \ j=1,2 \, .
\end{eqnarray}
\end{subequations}
Here $\nu(x)$ represents the external unit normal quadruple at $x
\in \partial\Omega$, and
\begin{subequations} \label{eq:formal1}
\begin{eqnarray}
L_{\ve,\alpha,\beta}(\phi,\psi^1,\psi^2) & = &
(\psi^1-\psi^2)\beta_2^{-1}
    H'_\ve(\phi)^* F'( P_\ve(\phi,\psi^1,\psi^2) )^*
    ( F(P_\ve(\phi,\psi^1,\psi^2)) - y^\delta ) \nonumber \\
& & \quad - \beta_1(2\beta_2)^{-1} H'_\ve(\phi) \,
    \nabla \cdot \big[ \nabla H_\ve(\phi) / |\nabla H_\ve (\phi)| \big] \, , \\
L^1_{\ve,\alpha,\beta}(\phi,\psi^1,\psi^2) & = & (2\beta_3)^{-1}
    \big( F'(P_\ve(\phi,\psi^1,\psi^2) )\, H_\ve(\phi) \big)^*
    ( F(P_\ve(\phi,\psi^1,\psi^2)) - y^\delta )  \\
L^2_{\ve,\alpha,\beta}(\phi,\psi^1,\psi^2) & = &  (2\beta_3)^{-1}
    \big( F'(P_\ve(\phi,\psi^1,\psi^2) )\, (1 - H_\ve(\phi)) \big)^*
    ( F(P_\ve(\phi,\psi^1,\psi^2)) - y^\delta ) \, .
\end{eqnarray}
\end{subequations}

It is worth noticing that the derivation of \eqref{eq:formal0} is
purely formal, since the $\bv$ seminorm is not differentiable.
Moreover the terms $|\nabla H_\ve (\phi)|$ and $|\nabla
(\psi^j-\psi^j_0)|$ appearing in the denominators of
\eqref{eq:formal0} and \eqref{eq:formal1}, respectively.

In Section~\ref{sec:numeric}, system \eqref{eq:formal0} and
\eqref{eq:formal1} is used as starting point for the derivation of a
level set type method.

\section{Inverse Elliptic Problems}\label{sec:numeric}

In this section, we discuss the proposed level set approach and
their application in some physical problems model by elliptic PDE's.
We also discuss briefly the numerical implementations of the
iterative method based on the level set approach. We remark that, in
the case of noise data the iterative algorithm derived by the level
set approach need an early stooping criteria \cite{EngHanNeu96}.

\subsection{The Inverse Potential Problem}

In this subsection, we apply the level set regularization framework
in an inverse potential problem \cite{DAP11, DCLT2010, HR96}.
Differently from \cite{DAP08, DAP11, DCLT2010, DCLT09, FSL05, DA06,
DAL09}, we assume that the source $u$ is not necessarily piecewise
constant. For relevant applications of the inverse potential problem
see \cite{HR96, Isa90, DAP11, DAP08} and references therein.

The forward problem consists of solving the Poisson boundary value
problem
\begin{equation} \label{eq:ipp}
-\nabla \cdot (\sigma \nabla   w) \ = \ u \, ,\ {\rm in} \ \Omega \,
, \quad \gamma_1 w + \gamma_2 w_\nu \ = \ g \, \ {\rm on} \
\partial\Omega\, ,
\end{equation}
on a given domain $\Omega \subset \mathbb R^n$ with $\partial
\Omega$ Lipschitz, for a given source function $u \in L^2(\Omega)$
and a boundary function $g \in L^2(\partial \Omega)$.
In \eqref{eq:ipp}, $\nu$ represent the outer normal vector to
$\partial \Omega$, $\sigma$ is a known sufficient smooth function.
Note that, depending of $\gamma_1, \gamma_2 \in \{0,1\}$, we have
Dirichlet, Neumann or Robin boundary condition. In this paper, we
only consider the case of Dirichlet boundary condition, that
corresponds to $\gamma_1 = 1$ and $\gamma_2 =0$ in \eqref{eq:ipp}.
Therefore, it is well known that there exists a unique solution $w
\in H^1(\Omega)$ of \eqref{eq:ipp} with $w- g \in H_0^1(\Omega)$,
\cite{DLions}.

Assuming homogeneous Dirichlet boundary condition in \eqref{eq:ipp},
the problem can be modeled by the operator equation
\begin{align}\label{eq:operator-ipp}
F_{1}\, :\, L^2(\Omega) &  \to L^2(\partial\Omega)\nonumber\\
\quad \quad \quad & u \longmapsto F_{1}(u) := w_{\nu}
|_{\partial\Omega}\,.
\end{align}

The corresponding inverse problem consists in recover the $L^2$
source function $u$, from measurements of the Cauchy data of its
corresponding potential $w$ on the boundary of $\Omega$.

Using this notation, the inverse potential problem can be written in
the abbreviated form $F_1(u) = y^\delta$, where the available noisy
data $y^\delta \in L^2(\partial\Omega)$ have the same meaning as in
\eqref{eq:noisy-data}. It is worth noticing that this inverse
problem has, in general, non unique solution \cite{HR96}. Therefore,
we restrict our attention to minimum-norm solutions
\cite{EngHanNeu96}. Sufficient conditions for identifiability are
given in \cite{Isa90}. Moreover, we restrict our attention to solve
the inverse problem \eqref{eq:operator-ipp} in $D(F)$, i.e., we
assume that the unknown parameter $u \in D(F)$, as defined in
Section~\ref{sec:min-concept}. Note that, in this situation, the
operator $F_1$ is linear. However, the inverse potential problem is
well known to be exponentially ill-posed \cite{Isa90}. Therefore,
the solution call for a regularization strategy \cite{EngHanNeu96,
HR96, Isa90}.

The following lemma implies that the operator $F_1$ satisfies the
Assumption {\bf(A2)}.
\begin{lemma}
The operator $F_1\,:\, D(F) \subset L^1(\Omega) \longrightarrow
L^2(\partial \Omega)$ is continuous with the respect to the
$L^1(\Omega)$ topology.
\end{lemma}
\begin{proof}
It is well known form the elliptic regularity theory \cite{DLions}
that $\norm{w}_{H^1(\Omega)} \leq c_1 \norm{u}_{L^2(\Omega)}$. Let
$u_n, u_0 \in D(F)$ and $w_n, w_0$ the respective solution of
\eqref{eq:ipp}. Then, the linearity and continuity of the trace
operator from $H^1(\Omega)$ to $L^2(\partial \Omega)$ \cite{DLions},
we have that
$$\norm{F_1(u_n) - F_1(u_0)}_{L^2(\partial \Omega)} \leq C \norm{w_n -
w_0}_{H^1(\Omega)}  \leq \tilde{C}  \norm{u_n -
u_0}_{L^2(\Omega)}\leq \tilde{C_1}  \norm{u_n -
u_0}_{L^1(\Omega)}\,,$$
where we use Lemma~\ref{lemma:0} to obtain the last inequality.
Therefore, $F_1$ is sequentially continuous on the $L^1(\Omega)$
topology. Since $L^1(\Omega)$ is a metrizable spaces
\cite{Yosida95}, the proof is complete.
\end{proof}

\subsubsection{A level set algorithm for the inverse potential problem}

We propose an explicit iterative algorithm derived from the
optimality conditions \eqref{eq:formal0} and \eqref{eq:formal1} for
the Tikhonov functional $\Ger$.

For the inverse potential problem with Dirichlet boundary condition
($\gamma_1=1 \mbox{ and } \gamma_2=0$) the algorithm reads as:
\begin{table}[bht!]
\fbox{ \begin{minipage}{0.96\textwidth} Given $\sigma$ and $g$;
\begin{tt}\begin{itemize}
\item[{\bf 1.}]   Evaluate the residual
$r_k := F_1( P_{\ve}(\phi_k,\psi_k^1,\psi_k^2) ) - y^\delta =
(w_k)_\nu |_{\partial\Omega} - y^\delta$, where $w_k$ solves

\centerline{\hfil $- \nabla\cdot(\sigma \nabla w_k) =
P_{\ve}(\phi_k,\psi_k^1,\psi_k^2)\, ,\ {\tt in}\ \Omega\, ; \qquad
\, w_k = g \, ,\ {\tt at}\
\partial\Omega \, .$ \hfil}

\item[{\bf 2.}]  Evaluate $h_k := F_1'(P_{\ve}(\phi_k,\psi_k^1,\psi_k^2))^* (r_k)
\in L^2(\Omega)$, solving \medskip

\centerline{\hfil $\Delta h_k = 0\, ,\ {\tt in}\ \Omega\, ; \qquad
h_k = r_k \, ,\ {\tt at}\ \partial\Omega \, .$ \hfil}

\item[{\bf 3.}] Calculate $\delta\phi_k := L_{\varepsilon,\alpha,\beta}(\phi_k,\psi_k^1,\psi_k^2)$
and $\delta \psi_k^j :=
L^j_{\varepsilon,\alpha,\beta}(\phi_k,\psi_k^1,\psi_k^2)$, as in
\eqref{eq:formal1}.

\item[{\bf 4.}] Update the level set function $\phi_k$ and the level values
$\psi_k^j$, $j = 1, 2$: \medskip

\centerline{\hfil $\phi_{k+1} = \phi_k + \frac{1}{\alpha} \;
\delta\phi_k \, ,
           \qquad \psi_{k+1}^j = \psi_k^j + \frac{1}{\alpha} \; \delta \psi_k^j \, .$ \hfil}
\end{itemize} \end{tt}
\end{minipage} }
\caption{Iterative algorithm based on the level set approach for the
inverse potential problem. \label{tab:sls} }
\end{table}

Each step of this iterative method consists of three parts (see
Table~\ref{tab:sls}): 1 - The residual $r_k \in L^2(\partial\Omega)$
of the iterate $(\phi_k, \psi_k^j)$ is evaluated (this requires
solving one elliptic BVP of Dirichlet type); 2 - The $L^2$--solution
$h_k$ of the adjoint problem for the residual is evaluated (this
corresponds to solving one elliptic BVP of Dirichlet type); 3 - The
update $\delta\phi_k$ for the level-set function and the updates
$\delta \psi_k^j$ for the level values are evaluated (this
corresponds to multiplying two functions).

In \cite{DAL09}, a level set method was proposed, where the
iteration is based on an inexact Newton type method. The inner
iteration is implemented using the conjugate gradient method.
Moreover, the regularization parameter $\alpha > 0$ is kept fixed.
In contrast to \cite{DAL09}, in Table~\ref{tab:sls}, we define
$\delta t = 1/\alpha$ (as a time increment) in order to derive an
evolution equation for the levelset function. Therefore, we are
looking for a fixed point equation related to the system of
optimality conditions for the Tikhonov functional. Moreover, the
iteration is based on a gradient type method as in \cite{DCLT2010}.

\subsection{The Inverse Problem in Nonlinear Electromagnetism}

Many interesting physical problems are model by quasi-linear
elliptic equations. Examples of applications include the
identification of inhomogeneity inside nonlinear magnetic materials
form indirect or local measurements. Electromagnetic non-destructive
tests aim to localize cracks or inhomogeneities in the steel
production process, where impurities can be described by a piecewise
smooth function, \cite{CBGI09, ChungChanTai05, BurOsher05, CVK10}.

In this section, we assume that $D \subset \subset \Omega$ is
measurable and
\begin{align}\label{eq:u}
u = \begin{cases}
     \psi_1 \,, x \in D\,,\\
     \psi_2 \,, x \in \Omega \ D\,,
\end{cases}
\end{align}
where $\psi_1, \psi_2 \in \B$ and $m>0$.

The forward problem consists of solving the Poisson boundary value
problem
\begin{equation} \label{eq:iep}
-\nabla \cdot (u \nabla   w) \ = \ f \, ,\ {\rm in} \ \Omega \, ,
\quad w \ = \ g \, \ {\rm on} \
\partial\Omega\,,
\end{equation}
where $\Omega \subset \mathbb R^n$ with $\partial \Omega$ Lipschitz,
the source $f \in H^{-1}(\Omega)$ and boundary condition $g \in
H^{1/2}(\partial \Omega)$.
It is well known that there exists a unique solution $w \in
H^1(\Omega)$ such that $w-g \in H_0^1(\Omega)$ for the
PDE~\eqref{eq:iep}, \cite{DLions}.

Assuming that during the production process the workpiece is
contaminated by impurities and that such impurities are described by
piecewise smooth function, the inverse electromagnetic problem
consist in the identification and the localization of the
inhomogeneities as well as the function values of the impurities.
The localization of support and the tabulation of the
inhomogeneities  values can indicate possible sources of
contamination in the magnetic material.

In other words, the inverse problem in electromagnetism consists in
the identification of the support (shape) and the function values of
$\psi^1, \psi^2$, of the coefficient function $u(x)$ defined in
\eqref{eq:u}. The voltage potential $g$ is chosen such that its
corresponding the current measurement $h := (w)_\nu
|_{\partial\Omega}$, available as a set of continuous measurement in
$\partial \Omega$. This problem is known in the literature as the
inverse problem for the Dirichlet-to-Neumann map \cite{Isa90}.

With this framework, the problem can be modeled by the operator
equation
\begin{align}\label{eq:operator-iep}
F_{2}\,:\, D(F) \subset L^1(\Omega) \to H^{1/2}(\partial\Omega)\nonumber\\
\quad \quad \quad u \longmapsto F_{2}(u) := w|_{\partial\Omega}\,,
\end{align}
where the potential profile $g = w|_{\partial\Omega} \in
H^{1/2}(\Omega)$ is given.

The authors in \cite{CVK10} investigated a level set approach for
solve an inverse problems of identification of inhomogeneities
inside a nonlinear material, from local measurements of the magnetic
induction. The assumption in \cite{CVK10} is that part of the
inhomogeneities are given by a crack localized inside the workpiece
and that outside the crack region,  magnetic conductivities are
nonlinear and they depends on the magnetic induction. In other
words, that $\psi_1 = \mu_1$ and $\psi_2 = \mu_2(|\nabla w|^2)$,
where $\mu_1$ is the (constant) air conductivity and
$\mu_2=\mu_2(|\nabla w|^2)$ is a nonlinear conductivity of the
workpiece material, whose values are assumed be known. In
\cite{CVK10}, they also present a successful iterative algorithm and
numerical experiment. However, in \cite{CVK10}, the measurements and
therefore the data are given in the whole $\Omega$. Such amount of
measurements are not reasonable in applications. Moreover, the
proposed level set algorithm is based on an optimality condition of
a least square functional with $H^1(\Omega)$-semi-norm
regularization. Therefore, there is no guarantee of existence of
minimum for the proposed functional.

\begin{remark}\label{remark:1}
Note that $F_2(u) = T_D(w)$, where $T_D$ is the Dirichlet trace
operator. Moreover, since $T_D: H^1(\Omega) \to H^{1/2}(\partial
\Omega)$ is linear and continuous \cite{DLions}, we have
$\norm{T_D(w)}_{H^{1/2}(\partial \Omega)} \leq c
\norm{w}_{H^1(\Omega)}$.
\end{remark}

In the following lemma,  we prove that the operator $F_2$ satisfies
the Assumption {\bf(A2)}.
\begin{lemma}
Let the operator $F_2\,:\, D(F) \subset L^1(\Omega) \longrightarrow
H^{1/2}(\partial \Omega)$ as defined in \eqref{eq:operator-iep}.
Then, $F_2$ is continuous with the respect to the $L^1(\Omega)$
topology.
\end{lemma}
\begin{proof}
Let $u_n, u_0 \in D(F)$ and $w_n, w_0$ denoting the respective
solution of \eqref{eq:ipp}. The linearity of equation \eqref{eq:iep}
implies that $w_n - w_0 \in H_0^1(\Omega)$ and it satisfies
\begin{align}\label{eq:wf}
\nabla \cdot ( u_n \nabla w_n) - \nabla \cdot ( u_0 \nabla w_0) =
0\,,
\end{align}
with homogeneous boundary condition. Therefore, using the weak
formulation for \eqref{eq:wf} we have
\begin{align*}
\int_\Omega \left( \nabla \cdot ( u_n \nabla w_n) - \nabla \cdot (
u_0 \nabla w_0) \right) \varphi dx = 0\,, \quad \forall \varphi \in
H^1_0(\Omega)\,.
\end{align*}
In particular, the weak formulation holds true for $\varphi = w_n -
w_0$. From the Green formula \cite{DLions} and the assumption that
$m>0$ (that guarantee  elipticity of \eqref{eq:iep}), follows that
\begin{align}\label{eq:est1}
m \norm{\nabla w_n - \nabla w_0}^2_{L^2(\Omega)} 
\leq \int_\Omega u_n |\nabla w_n - \nabla w_0|^2 dx 
\leq \int_\Omega |(u_n - u_0)| |\nabla w_0| |(\nabla w_n - \nabla
w_0)| dx\,.
\end{align}
From \cite[Theorem 1]{M63}, there exist $\varepsilon > 0$ (small
enough) such that $ w_0 \in W^{1,p}(\Omega)$ for $p = 2+
\varepsilon$. Using the H\"{o}lder inequality \cite{DLions} with
$1/p + 1/q = 1/2$ (note that $q> 2$ in the equation~\eqref{eq:est1},
follows that
\begin{align}\label{eq:a}
m \norm{\nabla w_n - \nabla w_0}^2_{L^2(\Omega)}  \leq \norm{u_n -
u_0}_{L^q(\Omega)} \norm{\nabla w_0}_{L^p(\Omega)} \norm{\nabla w_n
- \nabla w_0}_{L^2(\Omega)}\,.
\end{align}
Therefore, using the Poincar\'e inequality \cite{DLions} and
equation~\eqref{eq:a}, we have
\begin{align*}
\norm{ w_n -  w_0}_{H^1(\Omega)}  \leq C \norm{u_n -
u_0}_{L^q(\Omega)}\,,
\end{align*}
where the constant $C$ depends only of $m, \Omega,\norm{\nabla w_0}$
and the Poincar\'e constant. Now, the assertion follows from
Lemma~\ref{lemma:0} and Remark~\ref{remark:1}.
\end{proof}

\subsubsection{A level set algorithm for inverse problem in nonlinear electromagnetism}

We propose an explicit iterative algorithm derived from the
optimality conditions \eqref{eq:formal0} and \eqref{eq:formal1} for
the Tikhonov functional $\Ger$. Each iteration of this algorithm
consists in the following steps: In the first step the residual
vector $r\in L^2(\partial\Omega)$ corresponding to the iterate
$(\phi_n,\psi_n^1, \psi_n^2)$ is evaluated. This requires the
solution of one elliptic BVP's of Dirichlet type. In the second step
the solutions $v \in H^1(\Omega)$ of the adjoint problems for the
residual components $r$ are evaluated. This corresponds to solving
one elliptic BVP of Neumann type and to computing the inner-product
$ \nabla w \cdot \nabla v$ in $L^2(\Omega)$. Next, the computation
of $L_{\varepsilon,\alpha, \beta}(\phi_n,\psi_n^1, \psi_n^2)$ and
$L^j_{\varepsilon,\alpha, \beta}(\phi_n,\psi_n^1, \psi_n^2)$ as in
\eqref{eq:formal1}. The four step is the updates of the level-set
function $\delta\phi_n \in H^1(\Omega)$ and the level function
values $\delta\psi_n^j \in \bv(\Omega)$ by solve \eqref{eq:formal0}.

The algorithm is summarized in Table~\ref{tab:ls}.

\begin{table}[bht!]
\fbox{ \begin{minipage}{16.3cm}
\begin{tt}\begin{itemize}
\item[{\bf 1.}]  Evaluate the residual
$r:= F_2 (P_\varepsilon(\phi_n,\psi^1_n, \psi^2_n) ) - y^\delta =
 w|_{\partial\Omega} - g^\delta $, where $w \in H^1(\Omega) $
solves

\centerline{\hfil $\nabla \!\cdot\! ( P_\varepsilon(\phi_n,\psi_n^1,
\psi_n^2) \, \nabla w ) = f \, , \ {\tt in}\ \Omega\, ; \qquad  \, w
= g \, ,\ {\tt at}\ \partial\Omega \, .$ \hfil}

\item[{\bf 2.}]  Evaluate $F_2'(P_\varepsilon(\phi_n,\psi_n^1,\psi_n^2))^* \, r  :=
 \nabla w \cdot \nabla v \in L^2(\Omega)$,
where $w$ is the function computed in Step {\bf 1.} and $v \in
H^1(\Omega)$  solves 

\centerline{$\nabla \!\cdot\! ( P_\varepsilon(\phi_n,\psi_n^1,
\psi_n^2) \, \nabla v ) = 0 \ {\tt in}\ \Omega\, ; \qquad  \, v_\nu
= r \, ,\ {\tt at}\ \partial\Omega \, .$ \hfil}

\item[{\bf 3.}] Calculate $L_{\varepsilon,\alpha, \beta}(\phi_n,\psi_n^1,
\psi_n^2)$ and $L^j_{\varepsilon,\alpha, \beta}(\phi_n,\psi_n^1,
\psi_n^2)$ as in \eqref{eq:formal1}.

\item[{\bf 4.}] Evaluate the updates $\delta\phi \in H^1(\Omega)$, $\delta\psi^j \in \bv(\Omega)$ by
solving \eqref{eq:formal0} \medskip

\item[{\bf 5.}] Update the level set functions\,
$\phi_{n+1} = \phi_n + \frac{1}{\alpha}  \delta\phi$, \, and the
level function values $\psi_{n+1}^j = \psi_n^j + \frac{1}{\alpha}
\delta\psi^j$.
\end{itemize} \end{tt}
\end{minipage} }
\caption{An explicit algorithm based on the proposed level set
iterative regularization method. \label{tab:ls} }
\end{table}

\section{Conclusions and Future Directions}\label{sec:conclusions}

In this article, we generalize the results of convergence and
stability of the level set regularization approach proposed in
\cite{DCLT09, DCLT2010}, where the level values of discontinuities
are not piecewise constant inside of each region. We analyze the
particular case, where the set $\Omega$ is divide in two regions.
However, it is easy to extend the analysis for the case of multiple
regions adapting the multiple level set approach in \cite{CLT08,
DCLT09}.

We apply the level set framework for two problems: the inverse
potential problem and in an inverse problem in nonlinear
electromagnetism with piecewise non-constant solution. In both case,
we prove that the parameter-to-solution map satisfies the
Assumption~\textbf{(A1)}. The inverse potential problem application
is a natural generalization of the problem computed in \cite{CLT08,
DCLT2010, DCLT09}. We also investigate the applicability of an
inverse problem in nonlinear electromagnetism in the identification
of inhomogeneities inside a nonlinear magnetic workpiece. Moreover,
we propose iterative algorithm based on the optimality condition of
the smooth Tikhonov functional $\Ger$.

A natural continuation of this paper is the numerical
implementation. Level set numerical implementations for the inverse
potential problem was done before in \cite{CLT08, DCLT09, DCLT2010},
where the level values are assumed to be constant. Implementations
of level set methods for resistivity/conductivity problem in
elliptic equation  have been intensively implemented recently e.g.,
\cite{ChungChanTai05, DL09, SLD06, DA07, DA06, CVK10, BurOsher05}.

\section*{Acknowledgments}
A.D.C. acknowledges the partial support from IMEF - FURG.


\bibliographystyle{amsplain}
\bibliography{Final_v_123643_lset-const}

\end{document}